\documentclass{egpubl}
\usepackage{egsgp2020}

\SpecialIssuePaper         

\usepackage[T1]{fontenc}
\usepackage{dfadobe}
\usepackage[ruled,linesnumbered,longend]{algorithm2e}
\usepackage{color}
\usepackage{nicefrac}
\usepackage{microtype}
\usepackage{algpseudocode}
\usepackage{amsmath}

\usepackage{amsthm}
\usepackage{amsfonts}
\usepackage{mathrsfs}
\usepackage{txfonts}
\usepackage[noadjust]{cite}  
\usepackage{enumitem}

\DeclareMathOperator*{\argmin}{argmin}

\newcommand{\algocomment}[1]{{\tcp{#1}\vspace*{-0.25ex}}}
\newcommand{\algospace}{{\vspace*{0.3ex}}}
\theoremstyle{plain}
\newtheorem{theorem}{Theorem}[section]
\newtheorem{thm}{Theorem}[section]
\newtheorem{lemma}[thm]{Lemma}

\newtheorem{proposition}[thm]{Proposition}
\theoremstyle{definition}
\newtheorem{remark}[thm]{Remark}

\newtheorem{defn}[thm]{Definition}

\newcommand\prox[1]{\mathrm{prox}_{#1}}
\newcommand\meritfunc{\psi}
\newcommand\dre{\psi_{\mathrm{E}}}
\newcommand\drp{\psi_{\mathrm{P}}}
\newcommand\dom{\mathrm{dom}}
\newcommand\Rnn{\mathbb{R}^n}
\newcommand\dist{\mathrm{dist}}
\newcommand\Dg{\mathcal D_\gamma}
\newcommand\bf[1]{\mathbf{#1}}
\newcommand{\lev}{\mathrm{lev}}
\newcommand{\para}[1]{\vspace{1.0mm}\noindent\textbf{{#1}.}~~}
\newcommand{\imagefunc}[2]{{#2}_{#1}}		
\newcommand{\versionnote}{
	{
	\begin{picture}(0,0)(0,0)
	\put(-441,-428){ 
		\parbox{\textwidth}{%
		\footnotesize{
		This is the accepted version of the following article: \emph{Ouyang, W., Peng, Y., Yao, Y., Zhang, J. and Deng, B. (2020), Anderson Acceleration for Nonconvex ADMM Based on Douglas-Rachford Splitting. Computer Graphics Forum, 39(5)}, which has been published in final form at \url{http://onlinelibrary.wiley.com}. This article may be used for non-commercial purposes in accordance with the  \href{https://authorservices.wiley.com/author-resources/Journal-Authors/licensing/self-archiving.html}{Wiley Self-Archiving Policy}.}
		}
	}
	\end{picture}
	}
}		
\BibtexOrBiblatex
\electronicVersion
\PrintedOrElectronic
\ifpdf \usepackage[pdftex]{graphicx} \pdfcompresslevel=9
\else \usepackage[dvips]{graphicx} \fi

\usepackage{egweblnk} 
\graphicspath{{./FigsArXiv/}}

\title[Anderson Acceleration for Nonconvex ADMM Based on Douglas-Rachford Splitting]%
{Anderson Acceleration for Nonconvex ADMM\\ Based on Douglas-Rachford Splitting}

\author[W. Ouyang et al.]
{\parbox{\textwidth}{\centering Wenqing Ouyang$^{1}$
		\quad Yue Peng$^{1,2}$
		\quad Yuxin Yao$^{1}$
		\quad Juyong Zhang$^{1}$\thanks{Corresponding author: juyong@ustc.edu.cn (Juyong Zhang)}
		\quad Bailin Deng$^{2}$
	}
	\\
	{\parbox{\textwidth}{\centering $^1$University of Science and Technology of China \quad
			$^2$Cardiff University
		}
	}
}

\begin{document}
	
	
	\maketitle
	\begin{abstract}
		The alternating direction multiplier method (ADMM) is widely used in computer graphics for solving optimization problems that can be nonsmooth and nonconvex. It converges quickly to an approximate solution, but can take a long time to converge to a solution of high-accuracy. Previously, Anderson acceleration has been applied to ADMM, by treating it as a fixed-point iteration for the concatenation of the dual variables and a subset of the primal variables.
		In this paper, we note that the equivalence between ADMM and Douglas-Rachford splitting reveals that ADMM is in fact a fixed-point iteration in a lower-dimensional space. By applying Anderson acceleration to such lower-dimensional fixed-point iteration, we obtain a more effective approach for accelerating ADMM. 
		We analyze the convergence of the proposed acceleration method on nonconvex problems, and verify its effectiveness on a variety of computer graphics problems including geometry processing and physical simulation.
		\versionnote{}


	\end{abstract}  
	\section{Introduction}
Numerical optimization is commonly used in computer graphics, and finding a suitable solver is often instrumental to the performance of the algorithm. For an unconstrained problem with a simple smooth target function, gradient-based solvers such as gradient descent or the Newton method are popular choices~\cite{nocedal2006numerical}. 
On the other hand, for more complex problems, such as those with a nonsmooth target function or with nonlinear hard constraints, it is often necessary to employ more sophisticated optimization solvers to achieve the desired performance.
For example, proximal splitting methods~\cite{Combettes2011} are often used to handle nonsmooth optimization problems with or without constraints. The basic idea is to introduce auxiliary variables to replace some of the original variables in the target function, while enforcing consistency between the original variables and the auxiliary variables with a soft or hard constraint. This often allows to problem to be solved via alternating update of the variables, which reduces to simple sub-problems that can be solved efficiently. One example of such proximal splitting methods is the local-global solvers commonly used for geometry processing and physical simulation~\cite{SorkineA07,Liu2008,BouazizDSWP12,Liu2013,Bouaziz2014}.

Another popular type of proximal splitting methods, the alternating direction method of multipliers (ADMM)~\cite{boyd2011distributed}, is designed for the following form of optimization:
\begin{equation}
\min_{\mathbf{x},\mathbf{z}}~\varPhi(\mathbf{x}, \mathbf{z})\qquad
\textrm{s.t.}~\mathbf{A}\mathbf{x} - \mathbf{B} \mathbf{z} = \mathbf{c},
\label{eq:GeneralADMMProblem}
\end{equation}
where $\mathbf{x},\mathbf{z}$ are the original variable and the auxiliary variable, and the linear hard constraint $\mathbf{A}\mathbf{x} - \mathbf{B} \mathbf{z} = \mathbf{c}$ enforces their compatibility. ADMM computes a stationary point of the augmented Lagrangian function 
$L(\mathbf{x},\mathbf{z},\mathbf{y})=\varPhi(\mathbf{x}, \mathbf{z}) + \langle \beta \mathbf{y},\mathbf{A}\mathbf{x}-\mathbf{B}\mathbf{z}-\mathbf{c} \rangle + \frac{\beta}{2} \| \mathbf{A}\mathbf{x} - \mathbf{B} \mathbf{z} - \mathbf{c} \|^2$ via the following iterations~\cite{boyd2011distributed}:
\begin{align}
\mathbf{z}_{k+1}&=\argmin_{\mathbf{z}}~L(\mathbf{x}_{k}, \mathbf{z},\mathbf{y}_k), \label{eq:GeneralZStep}\\
\mathbf{x}_{k+1}&=\argmin_{\mathbf{x}}~L(\mathbf{x}, \mathbf{z}_{k+1},\mathbf{y}_k), \label{eq:GeneralXStep}\\
\mathbf{y}_{k+1}&=\mathbf{y}_k+\mathbf{A}\mathbf{x}_{k+1}-\mathbf{B}\mathbf{z}_{k+1}-\mathbf{c}, \label{eq:GeneralYStep}
\end{align}
where $\mathbf{y}$ is the dual variable, and $\beta \in \mathbb{R}^{+}$ is a penalty parameter.
This formulation is general enough to represent a large variety of optimization problems. For example, any additional hard constraint can be incorporated into the target function using an indicator function that vanishes if the constraint is satisfied and has value $+\infty$ otherwise. The above iteration often has a low computational cost, where each sub-problem can be solved in parallel and/or in a closed form. The solver can handle nonsmooth problems, and typically converges to an approximate solution in a small number of iterations~\cite{boyd2011distributed}. Moreover, although ADMM was initially designed for convex problems, it has proved to be also effective for many noncovex problems~\cite{wang2019global}. Such properties make it a popular solver for large-scale optimization in computer graphics~\cite{Neumann2013,Neumann2014-CMM,Overby2017}, computer vision~\cite{Lu2018-Unified,Wu2019-LpBox}, and image processing~\cite{Figueiredo2010,Almeida2013,Heide2016}.

Despite its popularity, a major drawback of ADMM is that it can take a long time to converge to a solution of high accuracy. This limitation has motivated various work on accelerating ADMM with a focus on convex problem~\cite{Goldstein2014,Kadkhodaie2015,Zhang2018-GMRES}. For nonconvex ADMM, an acceleration technique was proposed recently in~\cite{zhang2019accelerating}. By treating the steps~\eqref{eq:GeneralZStep}--\eqref{eq:GeneralYStep} as a fixed-point iteration of the variables $(\mathbf{x}, \mathbf{y})$ , it speeds up the convergence using Anderson acceleration~\cite{Anderson1965}, a well-known acceleration technique for fixed-point iterations.
It is also shown in~\cite{zhang2019accelerating} that for problems with a separable target function that satisfies certain assumptions, ADMM can be treated as a fixed-point iteration on a reduced set of variables, which further reduces the overhead of Anderson acceleration.

In this paper, we propose a novel acceleration technique for nonconvex ADMM from a different perspective. We note that if the target function is separable in $\mathbf{x}$ and $\mathbf{z}$, then ADMM is equivalent to Douglas-Rachford (DR) splitting~\cite{douglas1956numerical}, a classical proximal splitting method. Such equivalence enables us to interpret ADMM using its equivalent DR splitting form, which turns out to be a fixed-point iteration for a linear transformation of the ADMM variables, with the same dimensionality as the dual variable $\mathbf{y}$. As a result, we can apply Anderson acceleration to such alternative form of fixed-point iteration, often with a much lower dimensionality than the fixed-point iteration of $(\mathbf{x}, \mathbf{y})$ that is utilized in~\cite{zhang2019accelerating} for the general case and with a lower computational overhead. Moreover, compared to the other acceleration techniques in~\cite{zhang2019accelerating} based on reduced variables, our new approach has the same dimensionality for the fixed-point iteration but requires a much weaker assumption on the optimization problem. 
To achieve stability of the Anderson acceleration, we propose two merit functions for determining whether an accelerated iterate can be accepted: 1) the DR envelope, with a strong guarantee for global convergence of the accelerated solver, and 2) the primal residual norm, which provides fewer theoretical guarantees but incurs lower computational overhead. As far as we are aware of, this is the first global convergence proof for Anderson acceleration on nonconvex ADMM. 
We evaluate our method on a variety of nonconvex ADMM solvers used in computer graphics and other domains. Thanks to its low dimensionality and strong theoretical guarantee, our method achieves more effective acceleration than~\cite{zhang2019accelerating} on many of the experiments. 

To summarize, our main contributions include:
\begin{itemize}
	\item We propose an acceleration technique for nonconvex ADMM solvers, by utilizing their equivalence to DR splitting and applying Anderson acceleration to the fixed-point iteration form of DR splitting. We also propose two types of merit functions that can be used to verify the effectiveness of an accelerated iterate, as well as acceptance criteria for the iterate based on the merit functions.
	\item We prove the convergence of our accelerated solver under appropriate assumptions on the  problem and the algorithm parameters.
\end{itemize}

	\section{Related Works}

\para{ADMM} ADMM is a variant of the augmented Lagrangian scheme that uses partial updates for the dual variables, and is commonly used for optimization problems with separable target functions and linear side constraints~\cite{boyd2011distributed}. 
Its ability to handle nonsmooth and constrained problems and its fast convergence to an approximate solution makes it a popular choice for large-scale optimization in various problem domains. In computer graphics, ADMM has been applied for geometry processing~\cite{Bouaziz2013,Neumann2013,Zhang2014-LBC,Xiong2014,Neumann2014-CMM}, image processing~\cite{Heide2016}, computational photography~\cite{Wang2018-Megapixel}, and physical simulation~\cite{Gregson2014,Pan2017,Overby2017}.
It is well known that ADMM suffers from slow convergence to a high-accuracy solution, and different strategies have been proposed in the past to speed up its convergence, e.g., using Nesterov's acceleration~\cite{Goldstein2014,Kadkhodaie2015} or GMRES~\cite{Zhang2018-GMRES}. However, these acceleration methods focus on convex problems, while many problems in computer graphics are nonconvex. 

\para{Anderson Acceleration} Anderson acceleration~\cite{Anderson1965,Walker2011} is an established method for accelerating fixed-point iterations, and has been applied successfully to numerical solvers in different domains, such as numerical linear algebra~\cite{Sterck2012,Pratapa2016,Suryanarayana2016}, computational physics~\cite{Lipnikov2013,Willert2014,An2017,matveev2018anderson}, and robotics~\cite{PavlovODTO18}. The key idea of Anderson acceleration is to utilize $m$ previous iterates to construct a new iterate that converges faster to the fixed point. It has been noted that such an approach is indeed a quasi-Newton method~\cite{Eyert1996,Fang2009,Rohwedder2011}. Other research works have investigated its local convergence~\cite{Toth2015,Toth2017} as well as its effectiveness in acceleration~\cite{evans2020proof}. Recently, it has been applied in~\cite{Peng2018} to improve the convergence of local-global solvers in computer graphics. Later, Zhang et al.~\cite{zhang2019accelerating} proposed to speed up the convergence of nonconvex ADMM solvers in computer graphics using Anderson acceleration.

\para{DR Splitting} DR splitting was originally proposed in~\cite{douglas1956numerical} to solve differential equations for heat conduction problems, and has been primarily used for solving separable convex problems. In recent years, there is a growing research interest in its application on nonconvex problems~\cite{artacho2014recent,li2016douglas,phan2016linear,hesse2013nonconvex,hesse2014alternating}. The convergence of DR splitting in such scenarios has only been studied very recently~\cite{li2016douglas,themelis2020douglas}. In this paper, we will work with the same assumption as in~\cite{themelis2020douglas} to analyze the convergence of our algorithm.

Similar to ADMM, DR splitting also needs a large number of iterations to converge to a solution of high accuracy~\cite{fu2019anderson}. This has motivated research works on acceleration techniques for DR splitting, such as adaptive synchronization~\cite{bansode2019accelerated} and momentum acceleration~\cite{zhang2019achieving}. Anderson acceleration and similar adaptive acceleration strategies have also been used to accelerate DR splitting~\cite{fu2019anderson,poon2019trajectory}. However, these works consider convex problems only, and their convergence proofs rely heavily on the convexity. Thus they are not applicable to the nonconvex problems considered in this paper. 

The equivalence between ADMM and DR splitting is well known for convex problems~\cite{glowinski1983augmented}. Some existing methods utilize this connection to accelerate ADMM~\cite{pejcic2016accelerated,poon2019trajectory}, but they are only applicable to convex problems. 
Our method is based on the equivalence between ADMM and DR splitting for nonconvex problems, which has only been established very recently~\cite{bauschke2015projection,yan2016self,themelis2020douglas}.

	\section{Algorithm}
\label{sec:algorithm}
In this section, we first introduce the background for ADMM, DR splitting, and Anderson acceleration. Then we discuss the equivalence between ADMM and DR splitting on nonconvex problems, and derive an Anderson acceleration technique for ADMM based on its equivalent DR splitting form.

\subsection{Preliminary}
\para{ADMM}
 In this paper, we focus on ADMM for the following optimization problem with a separable target function:
\begin{equation}
\min\limits_{\bf{x},\bf{z}}~ f(\bf{x})+g(\bf{z})\qquad
\textrm{s.t.}~ \bf{A}\bf{x}-\bf{B}\bf{z}=\bf{c},
\label{eq:SeparableADMMProblem}
\end{equation}
with the ADMM steps given by:\begin{align}
\bf{x}_{k+1}&= \argmin_{\bf{x}} \Big(f(\bf{x})+\frac{\beta}{2}\|\bf{A}\bf{x}-\bf{B}\bf{z}_k+\bf{y}_k-\bf{c}\|^2\Big), \label{eq:SeparableADMMX}\\
\bf{y}_{k+1}&= \bf{y}_k+\bf{A}\bf{x}_{k+1}-\bf{B}\bf{z}_{k}-\bf{c}, \label{eq:SeparableADMMY}\\
\bf{z}_{k+1}&= \argmin_{\bf{z}} \Big(g(\bf{z})+\frac{\beta}{2}\|\bf{A}\bf{x}_{k+1}-\bf{B}\bf{z}+\bf{y}_{k+1}-\bf{c}\|^2\Big),  \label{eq:SeparableADMMZ} 
\end{align}
Throughout this paper, we assume that the solutions to sub-problems~\eqref{eq:SeparableADMMX} and \eqref{eq:SeparableADMMZ} always exist. Note that for each sub-problem, it is possible that there exist multiple solutions. Like~\cite{zhang2019accelerating}, we assume that the solver for each sub-problem is deterministic and always returns the same solution if given the same input, so that the operator $\argmin$ is single-valued.
Although the order of steps here appears different from the standard scheme in Eqs.~\eqref{eq:GeneralXStep}--\eqref{eq:GeneralYStep}, they are actually equivalent since they have the same relative order between the steps.
We adopt this notation instead of the standard scheme, because it facilitates our discussion about the equivalence with DR splitting. 
A commonly used convergence criterion for ADMM is that both the \emph{primal residual} $\mathbf{r}_{\textrm{p}}^k$ and the dual residual $\mathbf{r}_{\textrm{d}}^k$ vanish~\cite{boyd2011distributed}:
\begin{equation}
\mathbf{r}_{\textrm{p}}^k = \bf{A}\bf{x}_k-\bf{B}\bf{z}_{k-1}-\bf{c},\quad
\mathbf{r}_{\textrm{d}}^k = \beta \mathbf{B}^T \bf{A} (\bf x_{k} - \bf x_{k-1}).
\label{eq:primaldualresiduals}
\end{equation}
The primal and dual residuals  measure the violation of the linear side constraint and the dual feasibility condition of problem~\eqref{eq:SeparableADMMProblem}, respectively~\cite{boyd2011distributed}.
An alternative criterion is a vanishing \emph{combined residual}~\cite{Goldstein2014}:
\begin{equation} 
r_{\textrm{c}}^k=\beta\|\bf A\bf x_k-\bf B\bf z_{k-1}-\bf c\|^2+\beta\|\bf A(\bf x_{k}-\bf x_{k-1})\|^2,
\label{eq:combinedresiduals}
\end{equation}
which is a sufficient condition for vanishing primal and dual residuals. Moreover, the combined residual decreases monotonically for convex problems~\cite{Goldstein2014}.

\para{DR splitting} DR splitting has been used to solve optimization problems of the following form: 
\begin{align}
 \min\limits_{\bf{u}}&~~ \varphi_1(\bf{u})+\varphi_2(\bf{u}),
 \label{eq:DRProblem}
\end{align}
with an iteration scheme:
 \begin{align}
 \bf{s}_{k+1}&= \bf{s}_k+\bf{v}_k-\bf{u}_k, \label{eq:DRSStep}\\
 \bf{u}_{k+1}&= \prox{\gamma\varphi_1}(\bf{s}_{k+1}), \label{eq:DRUStep}\\
 \bf{v}_{k+1}&= \prox{\gamma\varphi_2}(2\bf{u}_{k+1}-\bf{s}_{k+1}), \label{eq:DRVStep}
 \end{align}
 where $\gamma \in \mathbb{R}^{+}$ is a constant and $\prox{h}$ denotes the proximal mapping of function $h$, i.e.,
 \begin{align}
 \label{eq:ProxMap}
 \prox{h}(\bf{x}):=\argmin_{\bf{y}\in\mathbb{R}^n} \Big(h(\bf{y})+\frac{1}{2}\|\bf{x}-\bf{y}\|^2\Big).
 \end{align}
Similar to our treatment of ADMM, we assume that there always exists a solution to the minimization problem above, and its solver always return the same result if given the same input, so that the proximal operator is single-valued.
Although DR splitting has been primarily used on convex optimization, recent results show that it is also effective for noncovex problems~\cite{li2016douglas}.
Later in Section~\ref{sec:DRBasedAA}, we will show that the ADMM steps \eqref{eq:SeparableADMMX}--\eqref{eq:SeparableADMMZ} are equivalent to the DR splitting scheme \eqref{eq:DRSStep}--\eqref{eq:DRVStep} for two functions $\varphi_1, \varphi_2$ derived from the target function and the linear constraint in Eq.~\eqref{eq:SeparableADMMProblem}.

\para{Anderson Acceleration} 
Given a fixed-point iteration
$$ \bf{x}_{k+1}={G}(\bf{x}_k), $$
Anderson acceleration~\cite{Anderson1965,Walker2011} aims at speeding up its convergence to a fixed point where the residual
$${F}(\bf{x})={G}(\bf{x})-\bf{x} $$
vanishes.
Its main idea is to use the residuals of the latest step $\mathbf{x}_{k}$ and its previous $m$ steps $\bf{x}_{k-1},..., \bf{x}_{k-m}$ to find a new step $\bf{x}_{k+1}^{AA}$ with a small residual. 
This is achieved via an affine combination of the images of $\bf{x}_{k}, \bf{x}_{k-1},..., \bf{x}_{k-m}$ under the fixed-point mapping $G$:
\[
	\mathbf{x}_{k+1}= G(\mathbf{x}_k)- \sum\limits_{j=1}^m \theta_j^{\ast} \left(G(\mathbf{x}_{k-j+1}) - G(\mathbf{x}_{k-j}) \right), 
\]
where the coefficients are found by solving a least-squares problem:
\[
	(\theta_1^\ast, \ldots, \theta_m^\ast) = \argmin_{\theta_1, \ldots, \theta_m} \left\|F(\mathbf{x}_k)- \sum_{j=1}^m \theta_j \left(F(\mathbf{x}_{k-j+1}) - F(\mathbf{x}_{k-j}) \right)\right\|^2.
\]

\subsection{Anderson Acceleration Based on DR Splitting}
\label{sec:DRBasedAA}
The derivation of our acceleration method relies on the equivalence between ADMM and DR splitting from~\cite{themelis2020douglas}, which we will review in the following. To facilitate the presentation, we first introduce a notation from~\cite{themelis2020douglas}:
\begin{defn}
	Given $f:\mathbb{R}^n \to {\mathbb{R}} \cup \{+\infty\}$ and $\bf{A} \in \mathbb{R}^{p \times n}$, the \emph{image function} $\imagefunc{\mathbf{A}}{f}: \mathbb{R}^{p} \to [-\infty, +\infty]$ is defined as
	$$ \imagefunc{\mathbf{A}}{f}(\bf{x})=
	\begin{cases}
	\inf_{\bf{y}}\{f(\bf{y})\mid\bf{A}(\bf{y})=\bf{x}\} & \textrm{if}~\mathbf{x}~\textrm{is in the range of}~\mathbf{A},\\
	+\infty &\textrm{otherwise}.
	\end{cases}
	    $$
\end{defn}
Note that we adopt a different symbol for image function than the one used in~\cite{themelis2020douglas} to improve readability.
The equivalence between ADMM and DR splitting is given as follows:
\begin{proposition}
\label{thm:Equivalence}
(\cite[Theorem 5.5]{themelis2020douglas}) Suppose $(\bf{x},\bf{y},\bf{z})\in\mathbb{R}^m\times\mathbb{R}^n\times\mathbb{R}^p$, and let $(\bf{x}^+,\bf{y}^+,\bf{z}^+)$ be generated by the ADMM iteration \eqref{eq:SeparableADMMX}--\eqref{eq:SeparableADMMZ} from $(\bf{x},\bf{y},\bf{z})$. Define
\begin{equation}
\left\{
\begin{aligned}
\bf{s} & = \bf{A}\bf{x}-\bf{y} \\
\bf{u} & = \bf{A}\bf{x} \\
\bf{v} & = \bf{B}\bf{z} + \bf{c}
\end{aligned}
\right.,
\qquad 
\left\{
\begin{aligned}
\bf{s}^+ & =  \bf{A}\bf{x}^+-\bf{y}^+ \\
\bf{u}^+ & =  \bf{A}\bf{x}^+ \\
\bf{v}^+ & =  \bf{B}\bf{z}^+ + \bf{c}
\end{aligned}
\right..
\label{eq:ADMM2DR}
\end{equation}
Then we have:
\begin{align}
\bf{s}^+ & = \bf{s}+\bf{v}-\bf{u}, \\
\bf{u}^+ & =  \prox{\gamma\varphi_1}(\bf{s}^+), \label{eq:DRuplus}\\
\bf{v}^+ & = \prox{\gamma\varphi_2}(2\bf{u}^+-\bf{s}^+), \label{eq:DRvplus}
\end{align}
where $\gamma = 1 / \beta$, and
\begin{equation}
	\varphi_1(\mathbf{u}) = \imagefunc{\mathbf{A}}{f}(\mathbf{u}),\quad
	\varphi_2(\mathbf{u}) = \imagefunc{\mathbf{B}}{g}(\mathbf{u}-\bf{c}).
	\label{eq:phi1phi2}
\end{equation} 
\end{proposition}
Proposition~\ref{thm:Equivalence} shows that for the optimization problem~\eqref{eq:SeparableADMMProblem}, we can find the functions $\varphi_1$ and $\varphi_2$ in the problem~\eqref{eq:DRProblem} such that the DR splitting steps~\eqref{eq:DRSStep}--\eqref{eq:DRVStep} are related to the ADMM steps~\eqref{eq:SeparableADMMX}--\eqref{eq:SeparableADMMZ} via the transformation defined in Eq.~\eqref{eq:ADMM2DR}. 

According to the DR splitting steps~\eqref{eq:DRUStep} and \eqref{eq:DRVStep}, both $\mathbf{u}_{k}$ and $\mathbf{v}_{k}$ are functions of $\mathbf{s}_{k}$. Then the step~\eqref{eq:DRSStep} indicates that $\mathbf{s}_{k+1}$ can be written as a function of $\mathbf{s}_k$ only:
\begin{equation}
\bf s_{k+1} = \mathcal{G}(\bf s_k):=\frac{1}{2}\left((2\prox{\gamma\varphi_2}-\bf I) \circ (2\prox{\gamma\varphi_1}-\bf I)+\bf I\right)(\bf s_k),
\label{eq:DRFixedPoint}
\end{equation}
where $\mathbf{I}$ denote the identity operator. In other words, the DR splitting steps can be considered as a fixed-point iteration of $\mathbf{s}$, which is a transformation of the variables $\mathbf{x}$ and $\mathbf{y}$ for its equivalent ADMM solver. Therefore, we can apply Anderson acceleration to the $\mathbf{s}$ variable in DR splitting to speed up the convergence. One tempting approach is to compute the value of $\mathbf{s}$ according to Eq.~\eqref{eq:ADMM2DR} after each ADMM iteration and apply Anderson acceleration. This would not work in general, however, because from an accelerated value of $\mathbf{s}$ we cannot recover the values of $\mathbf{x}$ and $\mathbf{y}$ to carry on the subsequent ADMM steps. Instead, we perform Anderson acceleration on DR splitting, and derive the ADMM solution $\mathbf{x}, \mathbf{y}, \mathbf{z}$ based on the final values of the DR splitting variables $\mathbf{s}, \mathbf{u}, \mathbf{v}$. To implement this idea, we still need to resolve a few problems. First, we need to determine the specific forms of the proximal operators $\prox{\gamma\varphi_1}$ and $\prox{\gamma\varphi_1}$ used in DR splitting. Second, similar to~\cite{zhang2019accelerating}, we need to define criteria for the acceptance of an accelerated iterate, to improve the stability of Anderson acceleration. Finally, we need to find a way to recover the ADMM variables $\mathbf{x}, \mathbf{y}, \mathbf{z}$ after the termination of DR splitting. These problems will be discussed in the following.

\subsubsection{Proximal Operators for $\gamma \varphi_1$ and $\gamma \varphi_2$} 
In general, given the functions $f$ and $g$ from the optimization problem~\eqref{eq:SeparableADMMProblem}, it is difficult to find an explicit formula for the image functions $\varphi_1$ and $\varphi_2$ given in Eq.~\eqref{eq:phi1phi2}. On the other hand, the proximal operators $\prox{\gamma \varphi_1}$ and $\prox{\gamma \varphi_2}$ have rather simple forms, as we will show below. Here and in the remaining parts of the paper, we will make frequent use of the following proposition from~\cite{themelis2020douglas}:
\begin{proposition}
\label{thm:prop5.2}
(\cite[Proposition 5.2]{themelis2020douglas}) 
Let $f:\mathbb{R}^n \to {\mathbb{R}} \cup \{+\infty\}$ and $\bf{A} \in \mathbb{R}^{p \times n}$. Suppose that for some $\beta > 0$ the set-valued mapping $\mathcal{X}_{\beta}(\mathbf{s}) := \argmin\limits_{\mathbf{x} \in \mathbb{R}^n}\{f(\mathbf{x}) + \frac{\beta}{2}\|\mathbf{A}\mathbf{x} - \mathbf{s}\|^2\}$ is nonempty for all $\mathbf{s} \in \mathbb{R}^p$. Then
\begin{enumerate}[label=(\roman*),align=left]
	\item the image function $f_{\mathbf{A}}$ is proper;
	\item $f_{\mathbf{A}}(\mathbf{A} \mathbf{x}_{\beta}) = f(\mathbf{x}_{\beta})$ for all $\mathbf{s} \in \mathbb{R}^p$ and $\mathbf{x}_{\beta} \in \mathcal{X}_{\beta}(\mathbf{s})$;
	\item $\prox{f_{\mathbf{A}}/\beta} = \mathbf{A} \mathcal{X}_{\beta}$.
\end{enumerate}
\end{proposition}
Then from Proposition~\ref{thm:prop5.2}, it is easy to derive the following:
\begin{proposition}
\label{prop:DRProximal}
The proximal operators $\prox{\gamma\varphi_1}, \prox{\gamma\varphi_2}$ defined in Eqs.~\eqref{eq:DRuplus} and \eqref{eq:DRvplus} can be evaluated as follows:
\begin{equation}
	\prox{\gamma\varphi_1}(\bf s) = \mathbf{A} \bar{\mathbf{x}}, \qquad \prox{\gamma\varphi_2}(2 \mathbf{u} - \bf s) = \mathbf{B} \overline{\mathbf{z}} + \mathbf{c},
	\label{eq:proxphi1phi2}
\end{equation}
where 
\begin{align}
\bar{\mathbf{x}} &= \argmin_{\mathbf{x}} \Big( f(\bf x) +\frac{1}{2\gamma}\|\bf A\bf x-\bf s\|^2 \Big), \label{eq:Xast}\\
\bar{\mathbf{z}} &= \argmin_{\mathbf{z}} \Big( g(\bf z)+\frac{1}{2\gamma}\|\bf B\bf z+\bf c - (2 \mathbf{u} - \bf s) \|^2 \Big). \label{eq:Zast}
\end{align}
\end{proposition}

\subsubsection{Criteria for Accepting Accelerated Iterate}
Classical Anderson acceleration can be unstable with slow convergence or stagnate at a wrong solution~\cite{Walker2011,Potra2013,peng2018anderson}.
To improve stability, in~\cite{zhang2019accelerating} an accelerated iterate is accepted only if it decreases a certain quantity that will converge to zero with effective iterations, such as the combined residual. Adopting a similar approach, we define a merit function $\meritfunc$ whose decrease indicates the effectiveness of an iteration. At the $k$-th iteration, we evaluate the un-accelerated iterate $\mathcal{G}(\mathbf{s}_{k-1})$ as well as the accelerated iterate $\mathbf{s}_{\textrm{AA}}$, and evaluate the decrease of the merit function from $\mathbf{s}_{k-1}$ to $\mathbf{s}_{\textrm{AA}}$:
\[
d = \meritfunc(\mathbf{s}_{\textrm{AA}}) - \meritfunc(\mathbf{s}_{k-1}).
\]
We choose $\mathbf{s}_{\textrm{AA}}$ as the new iterate if $d$ meets a certain criterion, and revert to the un-accelerated iterate $\mathcal{G}(\mathbf{s}_{k-1})$ otherwise.

One choice of the merit function is
\begin{equation}
	\drp(\mathbf{s}) := \| \mathbf{v}(\mathbf{s}) - \mathbf{u}(\mathbf{s}) \|,
	\label{eq:DRP}
\end{equation}
where $\mathbf{u}(\mathbf{s})$ and $\mathbf{v}(\mathbf{s})$ denote the $\mathbf{u}$ and $\mathbf{v}$ values produced by the DR splitting steps~\eqref{eq:DRUStep}  and \eqref{eq:DRVStep} from $\mathbf{s}$, i.e.,
\begin{equation}
\mathbf{u}(\mathbf{s}) = \prox{\gamma\varphi_1}(\bf{s}),
\qquad 
\mathbf{v}(\mathbf{s}) = \prox{\gamma\varphi_2}(2\bf{u}(\mathbf{s})-\bf{s}).
\label{eq:usvs} 
\end{equation}
Note that according to Eq.~\eqref{eq:DRSStep}, $\mathbf{v}(\mathbf{s}) - \mathbf{u}(\mathbf{s})$ measures the change in variable $\mathbf{s}$ between two consecutive iterations. Therefore, if $\mathbf{s}$ converges to a value $\mathbf{s}^{\ast}$, then $\drp(\mathbf{s})$ must converge to zero. Moreover, Proposition~\ref{thm:Equivalence} indicates that $\|\mathbf{v} - \mathbf{u}\| = \|\mathbf{A}\mathbf{x} - \mathbf{B}{\mathbf{z}} - \mathbf{c}\|$, which is the norm of the primal residual for the equivalent ADMM problem~\eqref{eq:SeparableADMMProblem}~\cite{boyd2011distributed}. We call $\drp(\mathbf{s})$ the \emph{primal residual norm}, and accept an accelerated iterate if its primal residual norm is no larger than the previous iterate. Thus the decrease criterion is:
\begin{equation}
d \leq 0.\\
\label{eq:DRPCondition}
\end{equation}

An alternative merit function is the \emph{DR envelope}:
\begin{equation}
 \dre(\bf{s}):=\min_{\bf{w}}\Big(\varphi_1(\bf{u}(\mathbf{s}))+\varphi_2(\bf{w})+\left\langle  \nabla\varphi_1(\bf{u}(\mathbf{s})), \bf{w}-\bf{u}(\mathbf{s}) \right\rangle+\frac{1}{2\gamma} \|\bf{w}-\bf{u}(\mathbf{s})\|^2 \Big),
 \label{eq:DRE}
\end{equation}
where $\mathbf{u}(\mathbf{s})$ is defined in Eq.~\eqref{eq:usvs}. It is shown in~\cite[Theorem 4.1]{themelis2020douglas} that $\dre(\bf{s})$ decreases monotonically during DR splitting iterations under the following assumptions:
	\begin{description}
		\item[(A.1)] $\varphi_1$ is $L$-smooth, $\sigma$-hypoconvex with $\sigma\in[-L,L]$.
		\item[(A.2)] $\varphi_2$ is lower semicontinuous and proper.
		\item[(A.3)] Problem \eqref{eq:DRProblem} has a solution.
	\end{description}
Here a function $F$ is said to be $L$-smooth if it is differentiable and $\|\nabla F(\bf x)-\nabla F(\bf y)\|\leq L\|\bf x-\bf y\|^2$~$\forall \bf x,\bf y$.
$F$ is said to be $\sigma$-hypoconvex if it is differentiable and 
$\langle \nabla F(\bf x)-\nabla F(\bf y),\mathbf{x} - \mathbf{y} \rangle\geq \sigma\|\bf x-\bf y\|^2$~$\forall \bf x,\bf y$.
$F$ is said to be lower semicontinuous if $\liminf\limits_{\bf x\rightarrow \bf{x}_0}F(\bf x)\geq F(\bf{x}_0)$~$\forall \bf{x}_0$.
$F$ is said to be proper if $F(\bf x)>-\infty~\forall\mathbf{x}$ and $F\nequiv+\infty$.
Under Assumptions (A.1)--(A.3), the DR envelope has a more simple form:
\begin{proposition}
	\label{prop1-9}
	If Assumptions (A.1)--(A.3) hold, then 
	\begin{equation}
		\dre(\bf s) = f(\bar{\mathbf{x}})+g(\bar{\mathbf{z}})+\frac{1}{\gamma}\langle \bf s-
		\bf u(\mathbf{s}),\bf v(\mathbf{s})-\bf u(\mathbf{s}) \rangle+\frac{1}{2\gamma}\|\bf v(\mathbf{s})-\bf u(\mathbf{s})\|^2,
		\label{eq:SimpleDRE}
	\end{equation}
	where $\bar{\mathbf{x}}, \bar{\mathbf{z}}$ are defined in~\eqref{eq:Xast} and \eqref{eq:Zast} respectively, and $\mathbf{u}(\mathbf{s}), \mathbf{v}(\mathbf{s})$ are defined in~\eqref{eq:usvs}. 
\end{proposition}
A proof is given in Appendix~\ref{app-A}. Note that the values $\bar{\mathbf{x}}, \bar{\mathbf{z}}, \mathbf{u}(\mathbf{s}), \mathbf{v}(\mathbf{s})$ are already evaluated during the DR splitting iteration. Therefore, the actual cost for computing $\dre(\bf s)$ is the evaluation of functions $f$ and $g$ as well as two inner products, which only incurs a small overhead in many cases.
Using the DR envelope as the merit function, we can enforce a more sophisticated decrease criterion that provides a stronger guarantee of convergence. Specifically, we require that $\mathbf{s}_{\textrm{AA}}$ decreases the DR envelope sufficiently compared to $\mathbf{s}_{k-1}$:
\begin{equation}
d \leq -\nu_1\|\mathcal{G}(\mathbf{s}_{k-1}) - \mathbf{s}_{k-1}\|^2-\nu_2\|\mathbf{s}_{\textrm{AA}} - \mathbf{s}_{k-1}\|^2,
\label{eq:DRECondition}
\end{equation}
where $\nu_1, \nu_2$ are nonnegative constants. The convergence of our solver using such acceptance criterion is discussed in Theorems~\ref{thm2-9} and \ref{thm4-4} in Section~\ref{sec:convergence}.

In this paper, unless stated otherwise, we use the DR envelope as the merit function to benefit from its convergence guarantee if the optimization problem satisfies the conditions given Theorems~\ref{thm2-9} or \ref{thm4-4}, and use the primal residual norm otherwise as it is an effective heuristic with lower overhead according to our experiments. 

\subsubsection{Recovery of $\mathbf{x}, \mathbf{y}, \mathbf{z}$} 
After the variable $\mathbf{s}$ converges to a fixed point $\mathbf{s}^\ast$ for the mapping $\mathcal{G}$, it is easy to recover the corresponding stationary point $(\mathbf{x}^\ast, \mathbf{y}^\ast, \mathbf{z}^\ast)$ for the ADMM problem. Before presenting the method, we first introduce the definition for the stationary points. 
\begin{defn}
	$(\bf{x}^*,\bf{y}^*,\bf{z}^*)$ is said to be a stationary point of~\eqref{eq:SeparableADMMProblem} if
	$$  \bf{A}\bf{x}^*-\mathbf{B}\bf{z}^*=\mathbf{c},\quad -\beta \bf{A}^T\bf{y}^*\in\partial f(\bf{x}^*),\quad \beta \bf{B}^T\bf{y}^*\in\partial g(\bf{z}^*),    $$
\end{defn}
where $\partial f$ and $\partial g$ denote the generalized subdifferentials of $f$ and $g$ \cite[Definition 8.3]{rockafellar2009variational}, respectively.
Our method for recovering $(\mathbf{x}^\ast, \mathbf{y}^\ast, \mathbf{z}^\ast)$  is based on the following:
\begin{proposition}
	\label{prop:StationaryPointRecovery}
	Let $\bf s^*$ be a fixed point of $\mathcal{G}$.  Define
	\begin{align*}
	\bf x^* &= \argmin_{\bf x}\Big(f(\bf x)+\frac{1}{2\gamma}\|\bf A\bf x-\bf s^*\|^2\Big) \\
	\bf u^* &= \mathbf{A} \mathbf{x}^*,\\
	\bf y^* &= \bf u^*-\bf s^*, \\
	\bf z^* &= \argmin_{\bf z}\Big(g(\bf z)+\frac{1}{2\gamma}\|\bf B\bf z+\bf c-(2\bf u^*-\bf s^*)\|^2\Big)  .
	\end{align*}
	Then $(\bf{x}^*,\bf{y}^*,\bf{z}^*)$ is a stationary point of the problem~\eqref{eq:SeparableADMMProblem}.
\end{proposition}
A proof is given in Appendix~\ref{app-B}. 
Note that the evaluation of $\mathbf{x}^\ast, \bf z^\ast$ has the same form as the intermediate values $\bar{\mathbf{x}}, \bar{\mathbf{z}}$ in Proposition~\ref{prop:DRProximal} for evaluating the proximal operators in DR splitting. Therefore, during the DR splitting, we store the values of $\bar{\mathbf{x}}$ and $\bar{\mathbf{z}}$ when evaluating the proximal operators. When the variable $\mathbf{s}$ converges, we simply return the latest values of $\bar{\mathbf{x}}, \bar{\mathbf{z}}$ as the solution to the ADMM problem.
Algorithm~\ref{algo1} summarizes our acceleration method.

\begin{algorithm}[t!]
        \caption{Anderson Acceleration for ADMM based on DR splitting.}
        \label{algo1}
            \KwData{
            	\hspace*{1ex}$\bf x_0,\bf y_0, \mathbf{z}_0$: initial values;\\
            	\hspace*{1ex}$m\in\mathbb{N}$: number of previous iterates used for acceleration;\\
				\hspace*{1ex}$k_{\max{}}$: maximum number of iterations;\\
				\hspace*{1ex}$\varepsilon$: convergence threshold.}
				$\mathbf{x}_{\textrm{default}} = \mathbf{x}_0$;~~$\mathbf{z}_{\textrm{default}} = \mathbf{z}_0$\;
            	$\bf s_0=\bf{A}\bf{x}_0-\bf y_0$;~~ $\mathbf{u}_0 = \mathbf{v}_0 = \mathbf{0}$;~~
        $\mathbf{s}_{\textrm{default}} = \mathbf{s}_0$\;
            	$k=0$;~~$\meritfunc_{\textrm{prev}} = r = + \infty$;~~ reset = TRUE;

            \While{TRUE}
            {
            	\algocomment{Perform one iteartion of DR splitting to evaluate merit function for $\mathbf{s}_k$}
            	$\bar{\mathbf{x}} =  \argmin_{\mathbf{x}} \Big(f(\bf{x})+\frac{1}{2\gamma}\|\bf{A}\bf{x}-{\mathbf{s}}_k\|^2\Big)$\;
            	$\bar{\mathbf{u}} = \mathbf{A} \bar{\mathbf{x}}$\;
            	$\bar{\mathbf{z}} = \argmin_{\mathbf{z}} \Big( g(\bf z)+\frac{1}{2\gamma}\|\bf B\bf z+\bf c - (2 \bar{\mathbf{u}} - \bar{\mathbf{s}}) \|^2 \Big)$\;
            	$\bar{\mathbf{v}}= \bf{B} \bar{\mathbf{z}}+\bf c$\;
            	Compute ${\meritfunc}$ using Eq.~\eqref{eq:DRP} (or Eq.~\eqref{eq:DRE})\;
            	$d = {\meritfunc{}} - \meritfunc_{\textrm{prev}}$\;
            	\algospace
            	\algocomment{Acceptance check for $\mathbf{s}_k$}
            	\eIf{ reset == TRUE \textbf{OR} $d$ satisfies condition~\eqref{eq:DRPCondition} (or~\eqref{eq:DRECondition}) }
            	{
            		\algocomment{Record the accepted iterate}
            		$\mathbf{x}_k = \mathbf{x}_{\textrm{default}} = \bar{\mathbf{x}}$; {~}
            		$\mathbf{z}_k = \mathbf{z}_{\textrm{default}} = \bar{\mathbf{z}}$; {~}
            		$\mathbf{u}_k = \mathbf{u}_{\textrm{default}} = \bar{\mathbf{u}}$\;
            		$\mathbf{v}_k = \mathbf{v}_{\textrm{default}} = \bar{\mathbf{v}}$; {~}
            		$\mathbf{s}_{\textrm{default}} = \mathbf{s}_{k}$\;
            		$\meritfunc_{\textrm{prev}} = {\meritfunc{}}$; {~~} reset = FALSE\;
            		\algospace{}
            		\algocomment{Compute accelerated iterate}
            		$\mathbf{g}_k= \mathbf{s}_{k} + \bar{\mathbf{v}} - \bar{\mathbf{u}}$;~
            		$\mathbf{f}_k= \mathbf{g}_k - \mathbf{s}_k$;~
            		$r = \|\mathbf{f}_k\|$;
            		~$\bar{m} = \min(m, k)$\;
            		$(\theta_1^\ast, \ldots, \theta_{\bar{m}}^\ast)  = \argmin\limits_{\theta_1, \ldots, \theta_{\bar{m}}} \left\|\mathbf{f}_k- \sum_{j=1}^{\bar{m}} \theta_j (\mathbf{f}_{k-j+1} - \mathbf{f}_{k-j})\right\|^2$\;
            		$\mathbf{s}_{\textrm{AA}} = \mathbf{g}_k- \sum\nolimits_{j=1}^{\bar{m}} \theta_j^{\ast} (\mathbf{g}_{k-j+1} - \mathbf{g}_{k-j})$\;
            		\algospace{}
            		\algocomment{Use $\mathbf{s}_{\textrm{AA}}$ for next acceptance check}
            		$\mathbf{s}_{k+1} = \mathbf{s}_{\textrm{AA}}$;~~$k = k+1$;
            	}
            	{
            		\algocomment{Revert to last accepted iterate}
            		$\mathbf{s}_{k} = \mathbf{s}_{\textrm{default}}$; {~}
            		$\mathbf{u}_{k} = \mathbf{u}_{\textrm{default}}$; {~}
            		$\mathbf{v}_{k} = \mathbf{v}_{\textrm{default}}$\;
            		$\mathbf{x}_{k} = \mathbf{x}_{\textrm{default}}$; {~}
            		$\mathbf{z}_{k} = \mathbf{z}_{\textrm{default}}$; {~}
            		reset = TRUE;
            	}
            	\algocomment{Check convergence}
            	\If{$k \geq k_{\max{}}$ \textbf{OR} $r < \varepsilon$}
            	{
            		\Return $\mathbf{x}_{\textrm{default}}$, $\mathbf{z}_{\textrm{default}}$;
            	}
        }
\end{algorithm}

\subsection{Discussion}
\subsubsection{Choice of Parameter $m$}
As pointed out in~\cite{Fang2009}, Anderson acceleration can be considered as a quasi-Newton method to find the root of the residual function, utilizing the $m$ previous iterates to approximate the inverse Jacobian. 
Similar to other Anderson acceleration based methods such as~\cite{higham2016anderson,peng2018anderson,zhang2019accelerating}, we observe that a larger $m$ leads to  more reduction in the number of iterations required for convergence, but also increases the overhead per iteration. We empirically set $m=6$ in all our experiments. 

\subsubsection{Comparison with~\cite{zhang2019accelerating}}
\cite{zhang2019accelerating} also proposed an Anderson acceleration approach for ADMM. In the general case, they treat the ADMM iteration~\eqref{eq:SeparableADMMX}--\eqref{eq:SeparableADMMZ} as a fixed-point iteration of $(\mathbf{x}, \mathbf{y})$. In comparison, Proposition~\ref{thm:Equivalence} shows that our approach is based on a fixed-point iteration of $\mathbf{s} = \mathbf{A}{\mathbf{x}} - \mathbf{y}$, with a dimensionality up to $50\%$ lower than $(\mathbf{x}, \mathbf{y})$. A main computational overhead for Anderson acceleration is  $2m$ inner products between vectors with the same dimensionality as the fixed-point iteration variables~\cite{Peng2018}. Therefore, our approach incurs a lower  overhead per iteration. The lower dimensionality of our formulation also indicates that it describes the inherent structure of ADMM in a more essential way. And we observe in experiments that such lower-dimensional representation can be more effective in reducing the number of iterations required for convergence. Together with the lower overhead per iteration, this often leads to faster convergence than the general approach from~\cite{zhang2019accelerating}.

It is also shown in~\cite{zhang2019accelerating} that if there is a special structure in the problem~\eqref{eq:SeparableADMMProblem}, ADMM can be represented as a fixed-point iteration of $\mathbf{x}$ or $\mathbf{y}$ alone, which would have the same dimensionality as the fixed-point mapping we use in this paper. In this case, besides the general approach mentioned in the previous paragraph, Anderson acceleration can also be applied to $\mathbf{x}$ or $\mathbf{y}$ alone, often with similar performance to our approach. However, this formulation requires one of the two target function terms in~\eqref{eq:SeparableADMMProblem} to be a strongly convex quadratic function, which is a strong assumption that limits its applicability. In comparison, our method imposes no special requirements on functions $f$ and $g$, making it a more versatile approach for effective acceleration.

	\section{Convergence Analysis}
\label{sec:convergence}

If we utilize the DR envelope as the merit function in Algorithm~\ref{algo1}, and use condition~\eqref{eq:DRECondition} to determine acceptance for an accelerated iterate, then it can be shown that Algorithm~\ref{algo1} converges to a stationary point to the optimization problem. In the following, we will discuss the conditions for such convergence.
Unless stated otherwise, we assume that all the functions are lower semicontinuous and proper.
In contrast to Section~\ref{sec:algorithm}, we will write $\in$ instead of $=$ for the evaluation of proximal mappings and minimization subproblems, to indicate that our results are still applicable when these operators are multi-valued.
We first introduce some definitions:
\begin{defn}
A point $\bf{s}^*$ is said to be a fixed point of the mapping $\mathcal{G}$ if $\bf{s}^*\in\mathcal{G}(\bf{s}^*)$.
\end{defn}
\begin{defn}
A point $\bf{u}^*$ is said to be a stationary point of~\eqref{eq:DRProblem} if
$$ 0\in \partial \varphi_1(\bf{u}^*)+\partial\varphi_2(\bf{u}^*).  $$
\end{defn}
\begin{defn}
A function $F$ is said to be level-bounded if the set $\{\bf x:F(\bf x)\leq \alpha\}$ is bounded for any $\alpha\in\mathbb{R}$. 
\end{defn}
Our first convergence result requires the following assumptions:
\begin{description}
\item[(B.1)] The constants $\nu_1,\nu_2$ in condition~\eqref{eq:DRECondition} satisfy $\nu_1>0,\nu_2\geq 0$.
\item[(B.2)] $\varphi_1+\varphi_2$ is level-bounded.
\item[(B.3)] The constant $\gamma = 1/\beta$ satisfies $\gamma<\min\{\frac{1}{2\max\{-\sigma,0\}},\frac{1}{L}\}$, where $L$ and $\sigma$ are defined in Assumption~(A.1).
\item[(B.4)] The function $\overline{g}(\mathbf{z}) := g(\mathbf{z}) + \frac{\beta}{2}\|\bf B \mathbf{z}+\bf c-\bf s\|^2$ is level-bounded and bounded from below for any given $\bf s$.
\end{description}
Our first convergence result is then given as follows:
\begin{thm}
\label{thm2-9}
Suppose Assumptions (A.1)--(A.3) and (B.1)--(B.3) hold. Let $\{(\bf{s}_k,\bf{u}_k,\bf{v}_k)\}$ be the sequence generated by Algorithm~\ref{algo1} using Eq.~\eqref{eq:DRECondition} as the acceptance condition. Then
\begin{itemize}
\item[(a)] $\{\dre(\bf{s}_k)\}$ is monotonically decreasing and $\|\bf{v}_k-\bf{u}_k\|\rightarrow 0$.
\item[(b)] The sequence $(\bf s_k,\bf u_k,\bf v_k)$ is bounded. If any subsequence $\{\bf s_{k_i}\}$ converges to a point $\bf{s}^*$,  then $\bf{s}^*$ is a fixed point of $\mathcal{G}$ and $\bf{u}^*=\prox{\gamma\varphi_1}(\bf{s}^*)$ is a stationary point of~\eqref{eq:DRProblem}. Moreover, such a convergent subsequence must exist.
\item[(c)] Suppose Assumption (B.4) is also satisfied. For any convergent subsequence $\{\bf s_{k_i}\}$ in (b),  let $\{\bf z_{k_i}\}$ be the corresponding subsequence generated by Algorithm~\ref{algo1}, i.e., 
\[\bf z_{k_i} \in \argmin_{\mathbf{z}} \Big( g(\bf z)+\frac{1}{2\gamma}\left\|\bf B\bf z+\bf c - (2 \mathbf{u}(\mathbf{s}_{k_i}) - \bf s_{k_i}) \right\|^2 \Big).\] Then $\{\bf z_{k_i}\}$ is bounded. Let $\bf{z}^*$ be a cluster point of $\{\bf z_{k_i}\}$, and define
\begin{equation*}
\bf{x}^*\in\argmin\limits_{\bf x}f(\bf x)+\frac{\beta}{2}\|\bf A\bf x-\bf{s}^*\|^2, \qquad  
\bf{y}^*=\bf{u}^*-\bf{s}^*. 
\end{equation*}
Then $(\bf{x}^*,\bf{y}^*,\bf{z}^*)$ is a stationary point of~\eqref{eq:SeparableADMMProblem}.
\end{itemize}
\end{thm}
A proof is given in Appendix~\ref{app-C}. 
\begin{remark}
Given a fixed point $\bf{s}^*$ of $\mathcal{G}$, we can also compute a stationary point~\eqref{eq:SeparableADMMProblem} without the assumptions used in Theorem~\ref{thm2-9}. The reader is referred to Appendix~\ref{app-E} for further discussion.
\end{remark}
 
Theorem~\ref{thm2-9} shows the subsequence convergence of $\{(\mathbf{s}_k, \mathbf{u}_k, \mathbf{v}_k)\}$ to a value corresponding to a stationary point. 
Next, we consider the global convergence of the whole sequence. 
We define $$\Dg(\bf s,\bf u,\bf v)=\varphi_1(\bf u)+\varphi_2(\bf v)+\frac{1}{\gamma}\langle \bf s-\bf u,\bf v-\bf u \rangle+\frac{1}{2\gamma}\|\bf v-\bf u\|^2.$$
Our global convergence results rely on the following assumptions:
\begin{description}
\item[(C.1)] The constants $\nu_1,\nu_2$ used in condition~\eqref{eq:DRECondition} are positive.
\item[(C.2)] Function $\Dg$ is sub-analytic.
\end{description}
The definition of a sub-analytic function can be found in~\cite{Xu2013}.
Then we can show the following:
\begin{thm}
\label{thm4-4}
Suppose assumptions (A.1)--(A.3), (B.1)--(B.3) and (C.1)--(C.2) hold. Let $\{(\bf{s}_k,\bf{u}_k,\bf{v}_k)\}$ be the sequence generated by Algorithm~\ref{algo1} using Eq.~\eqref{eq:DRECondition} as the acceptance condition. Then $\{(\bf{s}_k,\bf{u}_k,\bf{v}_k)\}$ converges to $(\bf{s}^*,\bf{u}^*,\bf{v}^*)$, where $\bf{s}^*$ is a fixed-point of $\mathcal{G}$, and $\bf{v}^*=\bf{u}^*=\prox{\gamma\varphi_1}(\bf{s}^*)$.
\end{thm}
A proof is given in Appendix~\ref{app-F}.
\begin{remark}
\label{remark4-5}
A sufficient condition for Assumption (C.2) is that $f$ and $g$ are both semi-algebraic functions. In this case, $\varphi_1$ and $\varphi_2$ will both be semi-algebraic~\cite{themelis2020douglas}, thus $\Dg$ is also semi-algebraic. Since a semi-algebraic function is also sub-analytic~\cite{Xu2013}, $\Dg$ will be a sub-analytic function. 
As noted in~\cite{zhang2019accelerating}, a large variety of functions used in computer graphics are semi-algebraic. Interested readers are referred to~\cite{zhang2019accelerating} and~\cite{li2015global} for further discussion.  
\end{remark}
\begin{remark}
\label{remark4-6}
If the functions $f$ and $g$ satisfy some further conditions, it can be shown that the convergence rate of $(\bf{s}_k,\bf{u}_k,\bf{v}_k)$ is r-linear. The discussion relies on the KL property~\cite{AttBolSva13} and is rather technical, so we leave it to Appendix~\ref{app-F}. 
\end{remark}
\begin{remark}
Assumption (A.1) requires the function  $f$ in~\eqref{eq:SeparableADMMProblem} to be globally Lipschitz differentiable. When $f$ is only locally Lipschitz differentiable, it is still possible to prove the convergence of Algorithm~\ref{algo1}. One such example is given in Appendix~\eqref{appx:PhysicalSimulation}.  
\end{remark}

\subsection{Assumptions on $f$ and $g$}
Assumptions (A.1), (A.2) and (B.2) impose conditions on the functions $\varphi_1$ and $\varphi_{2}$ in~\eqref{eq:DRProblem}. As there is no closed-form expression for $\varphi_1$ and $\varphi_2$ in general, these conditions can be difficult to verify. For practical purposes, we provide some conditions on the functions $f$ and $g$ that can ensure Assumptions (A.1), (A.2) and (B.2). These conditions are based on the results in \cite[Section 5.4]{themelis2020douglas}.
\begin{proposition}
\label{prop4-7}
	Suppose the problem~\eqref{eq:SeparableADMMProblem} and the ADMM sub-problems in~\eqref{eq:SeparableADMMX} and \eqref{eq:SeparableADMMZ} have a solution. 
	Then the following conditions are sufficient for Assumptions (A.1), (A.2) and (B.2):
	\begin{description}
		\item[(D.1)] $f$ and $g$ are proper and lower semicontinuous.
		\item[(D.2)] One of the functions $f$ and $g$ is level-bounded, and the other is bounded from below.
		\item[(D.3)] $\bf A$ is surjective.
		\item[(D.4)] $f$ satisfies one of the following conditions:
		\begin{enumerate}
			\item $f$ is Lipschitz differentiable, and $\argmin_{\bf x}\{f(\bf x)\mid\bf A\bf x=\bf s\}$ is single-valued and Lipschitz continuous;
			\item $f$ is Lipschitz differentiable and convex;
			\item $f$ is differentiable, and $\|\nabla f(\bf x)-\nabla f(\bf y)\|\leq L\|\bf A(\bf x-\bf y)\|^2$ for any $\mathbf{x},\mathbf{y}$ if $\nabla f(\bf x)$ and $\nabla f(\bf y)$ are in the range of $\bf A^T$.
		\end{enumerate}
		\item[(D.5)]The function $\mathcal{Z}(\mathbf{s}) := \argmin_{\mathbf{z}} \left\{g(\mathbf{z}) \mid \mathbf{B} \mathbf{z} + \mathbf{c} = \mathbf{s}\right\}$ is locally bounded on the set $\mathcal{S} = \{\mathbf{B}\mathbf{z} + \mathbf{c} \mid g(\mathbf{z}) < +\infty\}$, i.e., for any $\mathbf{s} \in \mathcal{S}$ there exists a neighborhood $O$ such that $\mathcal{Z}$ is bounded on $O$. 
	\end{description}
\end{proposition}
A proof is given in Appendix~\ref{app-G}.

	\section{Numerical Experiments}
\label{sec:Results}
We apply our method to a variety of problems to validate its effectiveness, focusing mainly on nonconvex problems in computer graphics. 
We describe each problem using the same variable names as in~\eqref{eq:SeparableADMMProblem}, so that its ADMM solver can be described by the steps~\eqref{eq:SeparableADMMX}--\eqref{eq:SeparableADMMZ}. 
Different solvers are run using the same initialization.
For each problem, we compare the convergence speed between the original ADMM solver, the accelerated solver (AA-ADMM) proposed in~\cite{zhang2019accelerating}, and our method. For each method we plot the combined residual~\eqref{eq:combinedresiduals} with respect to the iteration count and the computational time respectively, where a faster decrease of the combined residual indicates faster convergence. For ADMM and AA-ADMM, the combined residual is evaluated according to Eq.~\eqref{eq:combinedresiduals}. For DR splitting, it can be evaluated using the values of $\mathbf{s}, \mathbf{u}, \mathbf{v}$ without recovering their corresponding ADMM variables. Using the notations and results from Proposition~\ref{thm:Equivalence}, we have
\[
	\mathbf{u}^+ - \mathbf{v} = \mathbf{A} \mathbf{x}^+ - \mathbf{B} \mathbf{z} - \mathbf{c},
	\qquad \mathbf{u}^+ - \mathbf{u} =  \mathbf{A} (\mathbf{x}^+ - \mathbf{x}).
\]
Therefore, given an DR splitting iterate $(\mathbf{s}_k, \mathbf{u}_k, \mathbf{v}_k)$, we evaluate the combined residual $r_c^k$ by performing a partial iteration
\[
	{\mathbf{s}'} = \mathbf{s}_k + \mathbf{v}_k - \mathbf{u}_k, \qquad \mathbf{u}' = \prox{\gamma \varphi_1}({\mathbf{s}'})
\]
and computing
\[
	r_c^k = \frac{1}{\gamma}\left(\| \mathbf{u}' - \mathbf{v}_k \|^2 + \| \mathbf{u}' - \mathbf{u}_k \|^2\right).
\]
Similar to~\cite{zhang2019accelerating}, we normalize all combined residual values as follows to factor out the influence from the dimensionality and the value range of the variables:
\begin{equation}
	R = \sqrt{{r_c}\mathbin{/}({N_\mathbf{A} \cdot a^2})},
	\label{eq:Normalization}
\end{equation}
where $N_\mathbf{A}$ is the number of rows of matrix $\mathbf{A}$, and $a$ is a scalar that indicates the typical range of variable values. 
For both AA-ADMM and our method, we use $m=6$ previous iterates for Anderson acceleration.
We adopt the implementation of Anderson acceleration from~\cite{Peng2018}\footnote{\url{https://github.com/bldeng/AASolver}}.
All experiments are run on a desktop PC with a hexa-core CPU at 3.7GHz and 16GB of RAM.
The source codes for the examples are available at \url{https://github.com/YuePengUSTC/AADR}. 
 
\begin{figure}[t!]
	\centering
	\includegraphics[width=\columnwidth]{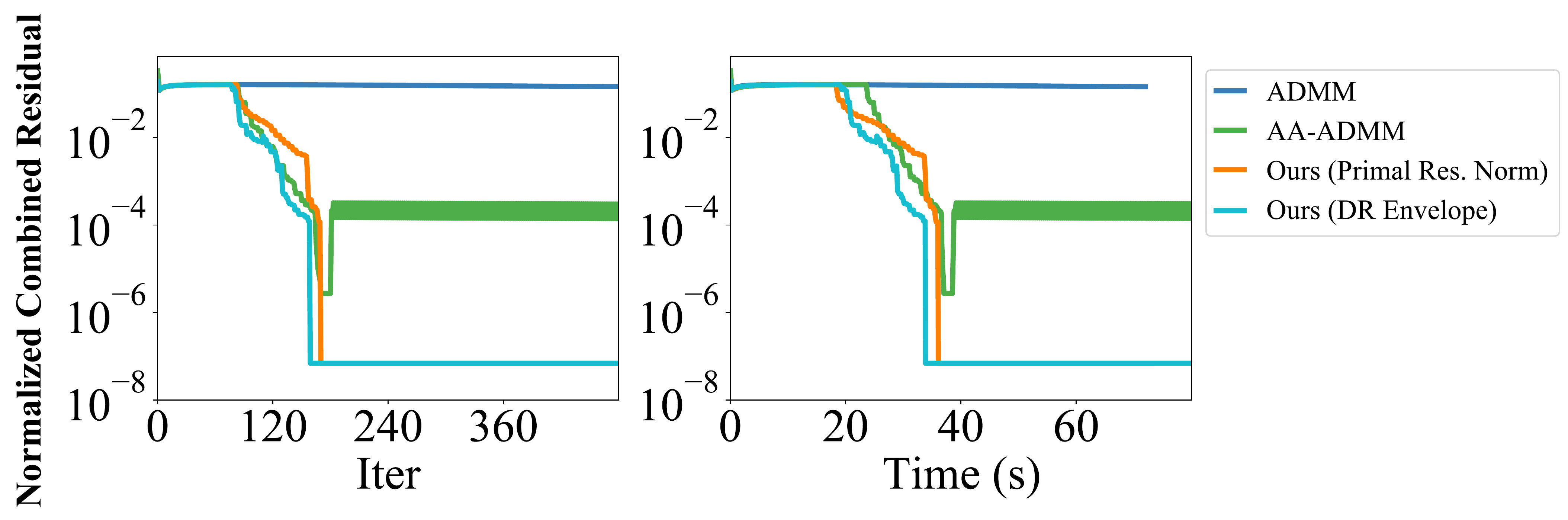}
	\caption{Comparison between ADMM, AA-ADMM, and our method with different merit functions, using the $\ell_q$-regularized logistic regression problem~\eqref{eq:LqRegresssion}. The two variants of our method have similar performance. Both accelerate the convergence of ADMM and perform better than AA-ADMM.}
	\label{fig:LogisticRegression}
\end{figure} 
 
\para{$\ell_q$-Regularized Logistic Regression} First, we consider a sparse logistic regression problem from the ADMM demo code for~\cite{wang2019global}\footnote{\url{https://github.com/shifwang/Nonconvex_ADMM_Demos}}: 
\begin{equation}
\min_{\mathbf{x},\mathbf{z}}~~  p\cdot \lambda \cdot \Omega(\mathbf{z}^1)+ \sum\nolimits_{i=1}^p \log(1 + \exp(-b_i(\mathbf{a}_i^T \mathbf{w} + v)))  \quad \textrm{s.t.}~\mathbf{x} = \mathbf{z}.
\label{eq:LqRegresssion}
\end{equation}
Here $\mathbf{x} = (\mathbf{w}, v)$ are the parameters to be optimized, with $\mathbf{w} \in \mathbb{R}^{n}$ and  $v \in \mathbb{R}$. $\mathbf{z} = (\mathbf{z}^1, z_2)$ is an auxiliary variable, with $\mathbf{z}^1 \in \mathbb{R}^n$ and ${z}^2 \in \mathbb{R}$. $\{(\mathbf{a}_i, b_i) \mid i = 1, \ldots, p\}$ is a set of input data pairs each consisting of a feature vector $\mathbf{a}_i \in \mathbb{R}^n$ and a label $b_i \in \{-1, 1\}$. $\Omega(\mathbf{z}^1) = \sum_{i=1}^n |z_i^1|^{1/2}$ is an $\ell_q$ sparsity regularization term with $q = \frac{1}{2}$. 
To test the performance, we use the data generator in the code to randomly generate $p = 1000$ pairs of data with feature vector dimension $n = 1000$. We test the problem with a weight parameter $\lambda = 10^{-4}$ and a penalty parameter $\beta = 10^5$. It can be verified that problem~\eqref{eq:LqRegresssion} satisfy the assumptions for Theorem~\ref{thm4-4} (see Appendix~\ref{appx:LqLogReg}). Thus we use the DR envelope as the merit function for Algorithm~\ref{algo1}, with parameter $\nu_1 = \nu_2 = 10^{-3}$ for the acceptance condition~\eqref{eq:DRECondition}. For comparison, we also run the algorithm using the primal residual norm as the merit function. We run AA-ADMM using the general approach in~\cite{zhang2019accelerating} that accelerates $\mathbf{x}$ and the dual variable $\mathbf{y}$ simultaneously, since the problem does not meet their requirement for reduced-variable acceleration. Fig.~\ref{fig:LogisticRegression} shows the comparison between the four solvers. We can see that both AA-ADMM and our methods can accelerate the convergence, while our methods achieve better performance thanks to the lower dimensionality of its accelerated variables. In addition, there is no significant difference between the performance of the two variants of our method, which verifies the effectiveness of the primal residual norm as the merit function despite its lack of convergence guarantee in theory.

\begin{figure}[t!]
	\centering
	\includegraphics[width=\columnwidth]{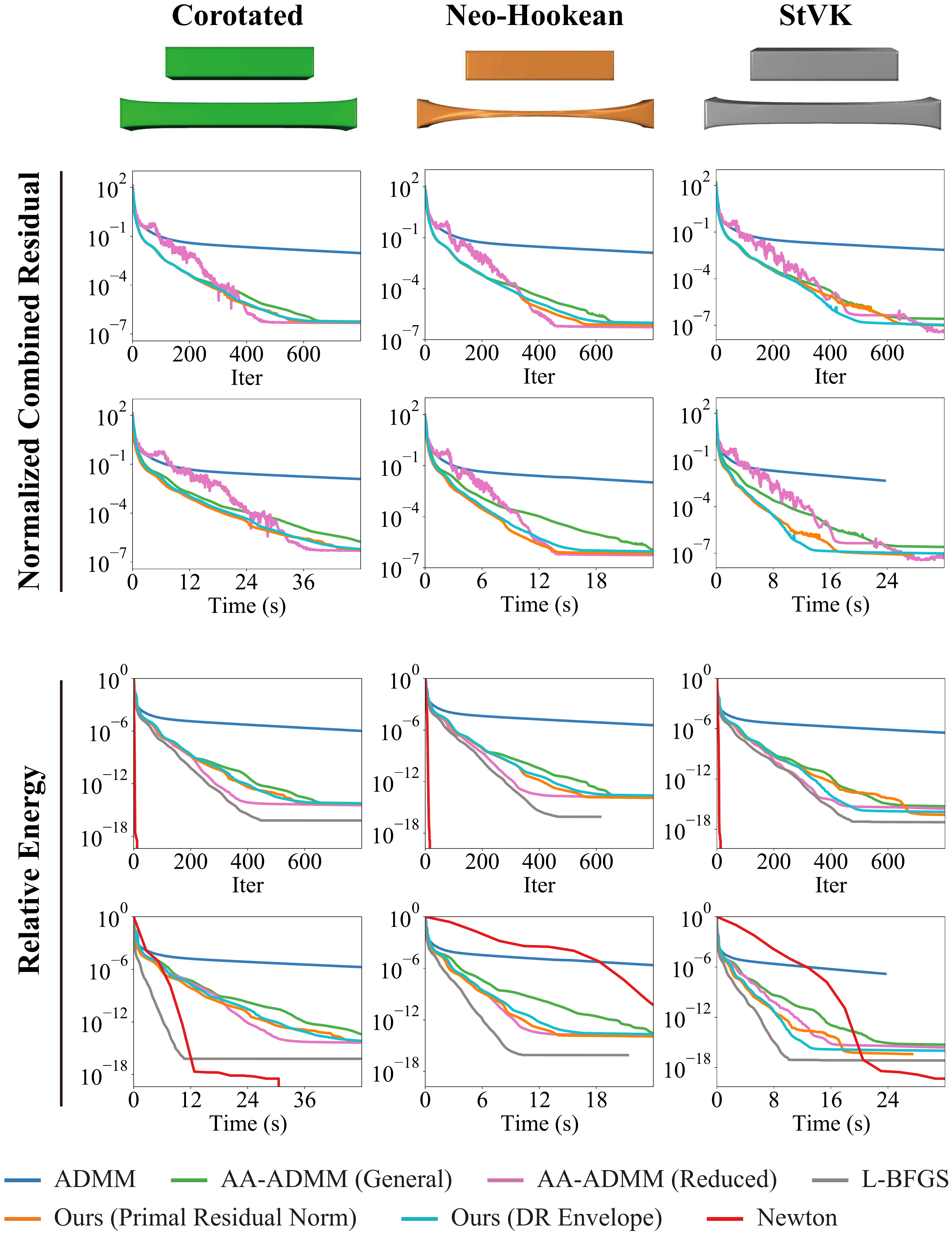}
	\caption{Comparison using~\eqref{eq:OverbyProblem} for computing a frame in physical simulation of a stretched elastic bar with 6171 vertices and 25000 tetrahedrons, using three types of hyperelastic energy and a high stiffness parameter (`rubber' in the source code of~\cite{Overby2017}). The normalized combined residual plots (the top two rows) show that both variants of our method achieve similar acceleration results as the reduced-variable scheme of AA-ADMM. All three approaches perform better than the general scheme of AA-ADMM. The bottom two rows plot the relative energy~\eqref{eq:PhysicalSimulationRelativeEnergy} and include a Newton solver~\cite{Sifakis2012} and an L-BFGS solver for~\cite{LiuBK17} for comparison.}
	\label{fig:SimulationSchemeCompare}
\end{figure} 
\begin{figure}[t!]
	\centering
	\includegraphics[width=\columnwidth]{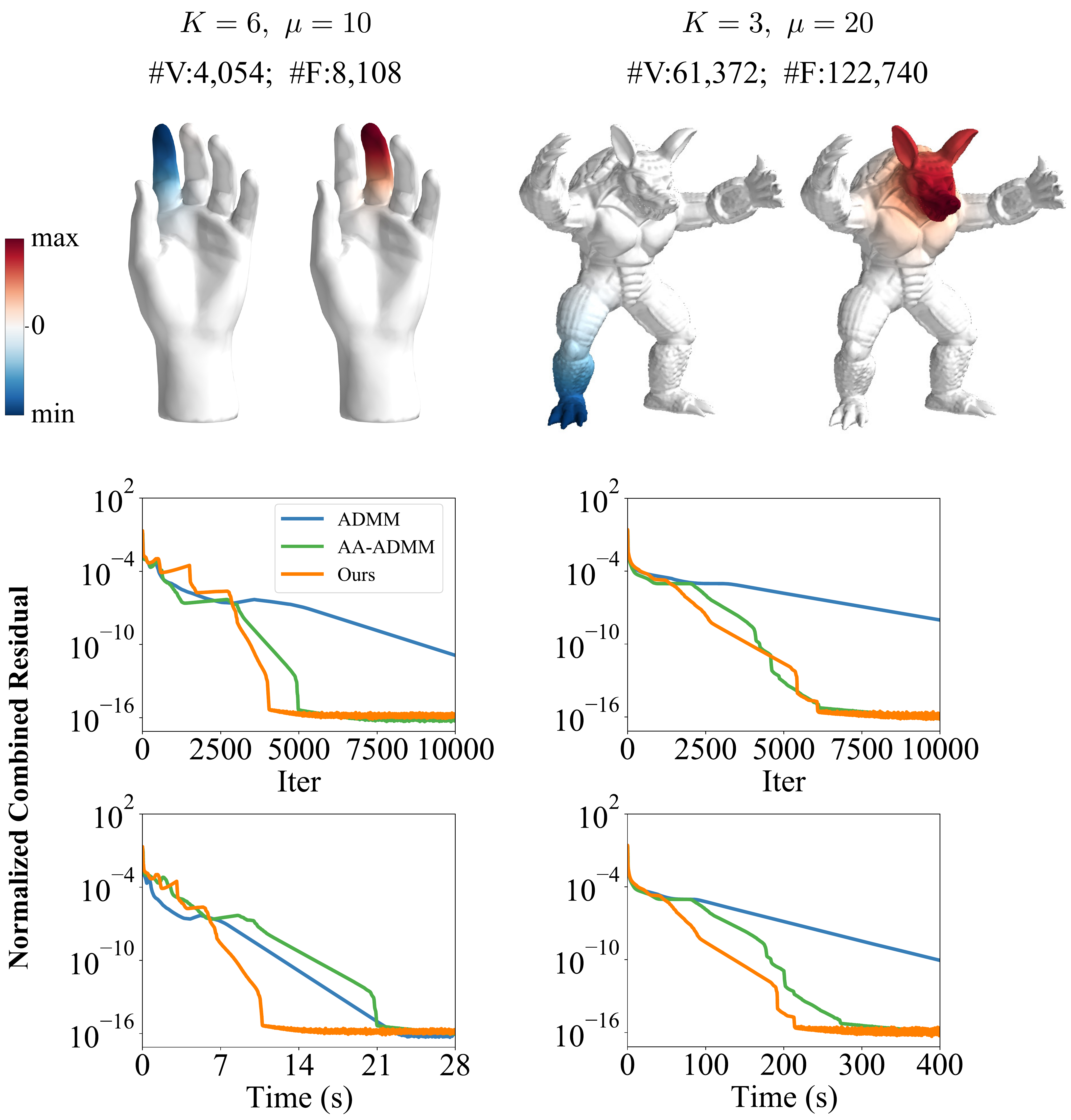}
	\caption{Computation of compressed manifold basis via problem~\eqref{eq:CMMProblem}. Our method achieves similar reduction of iterations as AA-ADMM, but outperforms AA-ADMM in computational time thanks to its lower overhead.}
	\label{fig:CompressedManifoldModes}
\end{figure}

\begin{figure*}[t!]
	\centering
	\includegraphics[width=\textwidth]{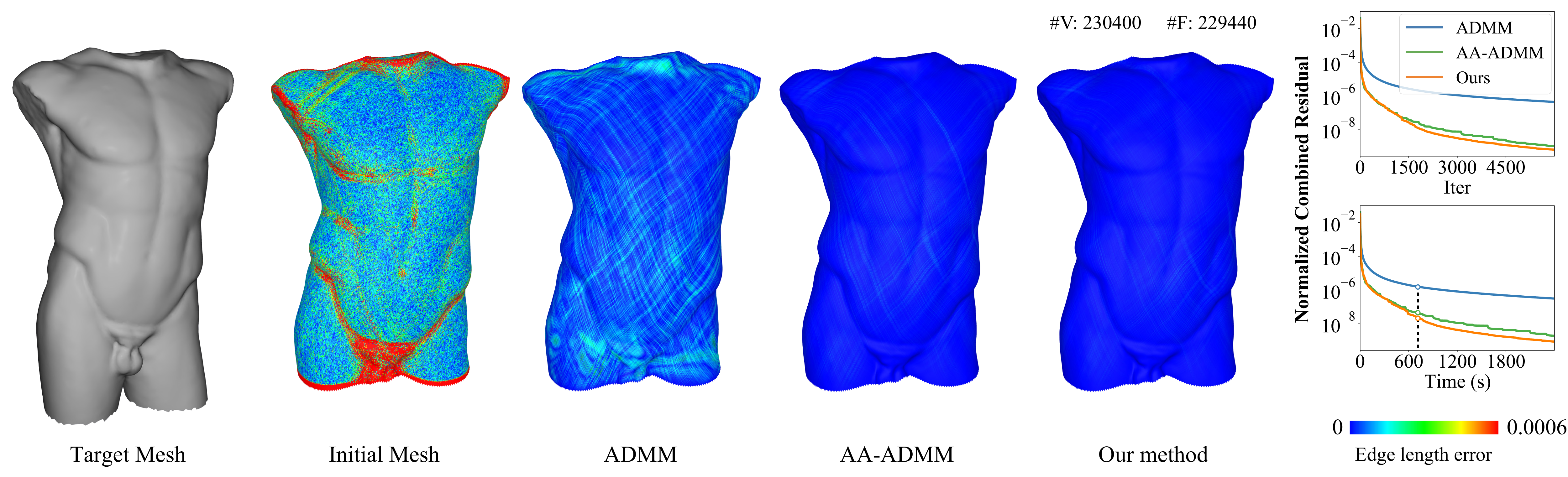}
	\caption{Comparison between ADMM and accelerated methods on a wire mesh optimization problem~\eqref{eq:GeometryOptProblem2}. The normalized combined residual plots show faster convergence using the accelerated solvers and better performance with our method. The color-coding visualizes the edge length error $\xi$ defined in~\eqref{eq:EdgeLengthError} on meshes computed by the three methods within the same computational time (see the bottom-right plot).}
	\label{fig:WireMesh}
\end{figure*}

\para{Physical Simulation} Next, we consider the ADMM solver used in~\cite{Overby2017} for the following optimization for physical simulation: 
\begin{equation}
\min_{\mathbf{x},\mathbf{z}} ~~ f(\mathbf{x}) + g(\mathbf{z})\quad
\textrm{s.t.} ~~ \mathbf{W}(\mathbf{x}-\mathbf{D} \mathbf{z} ) = 0,
\label{eq:OverbyProblem}
\end{equation}
where $\mathbf{z}$ is the node positions to be optimized, $\mathbf{x}$ is an auxiliary variable that represents the absolute or relative node positions for the elements according to the selection matrix $\mathbf{D}$, $\mathbf{W}$ is a diagonal weight matrix, $f$ is an elastic potential energy, and $g$ is a quadratic momentum energy. 
In Appendix~\ref{appx:PhysicalSimulation}, we use the StVK model as an example to prove the convergence of Algorithm~\ref{algo1} on problem~\eqref{eq:OverbyProblem}.
AA-ADMM can be applied to this problem to accelerate the variable $\mathbf{x}$ alone~\cite{zhang2019accelerating}, and we include both the general approach and the reduced-variable approach for comparison. For our method, we include the implementation using each merit function into the comparison, and choose parameter $\nu_1 = \nu_2 = 0$ for the acceptance condition~\eqref{eq:DRECondition}. 
Fig.~\ref{fig:SimulationSchemeCompare} shows the performance of the five solver variants on the simulation of a stretched hyperelastic bar with a high stiffness parameter, using three types of hyperelastic energy. 
We adapt the source codes from~\cite{Overby2017}\footnote{\url{https://github.com/mattoverby/admm-elastic}} and \cite{zhang2019accelerating}\footnote{\url{https://github.com/bldeng/AA-ADMM}} for the implementation of ADMM and AA-ADMM, respectively.
The normalized combined residual plots (the top two rows) show that all accelerated variants achieve better performance than the ADMM solver. Overall, the general AA-ADMM takes a long time than other accelerated variants for full convergence, potentially due to the larger number of variables involved in the fixed-point iteration and the higher overhead they induce.
For a more complete evaluation, we also compare the solvers with a Newton method~\cite{Sifakis2012} and an L-BFGS method~\cite{LiuBK17}, neither of which suffers from slow final convergence. Specifically, we use them to minimize the following energy equivalent to the target function of~\eqref{eq:OverbyProblem}:
\begin{equation}
	F(\mathbf{z}) = f(\mathbf{D} \mathbf{z}) + g(\mathbf{z}).
	\label{eq:PhysicalSimulationEnergy}
\end{equation}
In the bottom two rows of Fig.~\ref{fig:SimulationSchemeCompare}, we compare  all methods by plotting their relative energy 
\begin{equation}
	E = (F - F^\ast) / (F_0 - F^\ast),
	\label{eq:PhysicalSimulationRelativeEnergy}
\end{equation}
with respect to the iteration count and computational time, where $F_0$ and $F^\ast$ are the initial value and the minimum of the energy $F$, respectively. We can see that although the Newton method requires the fewest iterations to convergence, it is one of the slowest methods in terms of actual computational time, due to its high computational cost per iteration. L-BFGS achieves the best performance in terms of computational time, followed by the accelerated ADMM solvers. Note, however, that classical Newton and L-BFGS are intended for smooth unconstrained optimization problems, and they are often not applicable if the problem is nonsmooth or constrained --- the type of problems that ADMM is popular for.

\para{Geometry Processing} Nonconvex ADMM solvers have also been used in geometry processing. In Fig.~\ref{fig:CompressedManifoldModes}, we compare the performance between different methods on the following optimization problem from~\cite{Neumann2014-CMM} for compressed manifold modes on a triangle mesh with $N$ vertices:
\begin{equation}
\min_{\mathbf{X}, \mathbf{Z}}~ \text{Tr}((\mathbf{X}^1)^T \mathbf{L} \mathbf{X}^1)+\mu \|\mathbf{X}^2\|_1 + \iota(\bf Z) \quad
\textrm{s.t.}~\mathbf{Z}=\mathbf{X}^1, \mathbf{Z} = \mathbf{X}^2,
\label{eq:CMMProblem}
\end{equation}
where $\mathbf{Z} \in \mathbb{R}^{N\times K}$ denotes a set of basis functions to be optimized, $\mathbf{X}^1, \mathbf{X}^2 \in \mathbb{R}^{N\times K}$ are auxiliary variables, $\bf L \in \mathbb{R}^{N \times N}$ is a Laplacian matrix, and $\iota$ is an indicator function of $\mathbf{Z}$ for enforcing the orthogonality condition $\textrm{if}~\mathbf{Z}^T \mathbf{D} \mathbf{Z} = \mathbf{I}$ with respect to a mass matrix $\mathbf{D}$.
We apply our method with the primal residual norm as the merit function. 
We use the source code released by the authors\footnote{\url{https://github.com/tneumann/cmm}} for the ADMM solver, and modify it to implement AA-ADMM and our method. We use the general approach of AA-ADMM that accelerates $\mathbf{X}$ together with the dual variable, as the problem does not meet the requirement for reduced-variable acceleration. Fig.~\ref{fig:CompressedManifoldModes} shows the combined residual plots for the three methods on two models as well as the parameter settings for each problem instance. Our method achieves a similar effect in reducing the number of iterations as AA-ADMM, but outperforms AA-ADMM in terms of computational time thanks to its lower computational overhead.

We also apply our method to a problem proposed in~\cite{Deng2015} for optimizing the vertex positions $\mathbf{x} \in \mathbb{R}^{3n}$ of a mesh model subject to a set of soft constraints $\mathbf{A}_i \mathbf{x} \in \mathcal{C}_i$ ($i \in \mathcal{S}$) and hard constraints $\mathbf{A}_j \mathbf{x} \in \mathcal{C}_j$ ($j \in \mathcal{H}$),
where matrices $\mathbf{A}_i$ and $\mathbf{A}_j$ select the relevant vertices for the constraints and compute their differential coordinates where appropriate, and $\mathcal{C}_i$ and $\mathcal{C}_j$ represent the feasible sets. This is formulated in~\cite{Deng2015} as the following optimization:
\begin{align}
\min_{\mathbf{x}, \mathbf{z}}~~ &\frac{1}{2}\left\|\mathbf{L}(\mathbf{x} - \tilde{\mathbf{x}})\right\|^2 + \sum_{i \in \mathcal{S}} \left(\frac{w_i}{2} \|\mathbf{A}_i \mathbf{x} - \mathbf{z}_i\|^2 +  \sigma_{\mathcal{C}_i}(\mathbf{z}_i) \right) + 
\sum_{j \in \mathcal{H}} \sigma_{\mathcal{C}_j}(\mathbf{z}_j)\nonumber\\
\textrm{s.t.}~~& \mathbf{A}_j {\mathbf{x}} - \mathbf{z}_j = \mathbf{0}~~\forall j \in \mathcal{H}.
\label{eq:GeometryOptimizatoinProblem}
\end{align}
where $\mathbf{z}_i$ ($i \in \mathcal{S}$) and $\mathbf{z}_j$ ($j \in \mathcal{H}$) are auxiliary variables,  $\sigma_{\mathcal{C}_i}$ and $\sigma_{\mathcal{C}_j}$ are indicator functions for the feasible sets, and $w_i$ are user-specified weights. The first term of the target function is an optional Laplacian smoothness energy, whereas the second term measures the violation of the soft constraints using the squared Euclidean distance to the feasible sets. 
This problem is solved using ADMM and AA-ADMM in~\cite{zhang2019accelerating}. However, since its target function is not separable, our accelerated ADMM solver is not applicable. To apply our method, we reformulate the problem as follows:
\begin{align}
\min_{\mathbf{x}, \mathbf{z}} &~~\frac{1}{2}\left\|\mathbf{L}(\mathbf{x} - \tilde{\mathbf{x}})\right\|^2 + \sum_{i \in \mathcal{S}} \frac{w_i}{2} \left(D_{\mathcal{C}_i}(\mathbf{z}_i)\right)^2  + \sum_{j \in \mathcal{H}} \sigma_{\mathcal{C}_j} (\mathbf{z}_j)\nonumber\\
\textrm{s.t.} &~~\mathbf{A}_i \mathbf{x} = \mathbf{z}_i ~~\forall i \in \mathcal{S},
\quad \mathbf{A}_j \mathbf{x} = \mathbf{z}_j ~~\forall j \in \mathcal{H},
\label{eq:GeometryOptProblem2}
\end{align}
where $D_{\mathcal{C}_i}(\cdot)$ denotes the Euclidean distance to $\mathcal{C}_i$.
This problem has a separable target function, and we derive its ADMM solver in Appendix~\ref{appx:GeometryOpt}. 
We compare the performance of ADMM, AA-ADMM and our method on problem~\eqref{eq:GeometryOptProblem2} for wire mesh optimization~\cite{Grinspun14b}: we optimize a regular quad mesh subject to the soft constraints that each vertex lies on a target shape, and the hard constraints that (1)~each edge has the same length $l$ and (2)~all angles of each quad face are within the range $[\pi/4, 3 \pi/4]$. In Fig.~\ref{fig:WireMesh}, We solve the problem on a mesh with 230K vertices, using $\mathbf{L} = \mathbf{0}$, $w_i = 1$, and penalty parameter $\beta = 10000$. The combined residual plots show that both AA-ADMM and our method and our method achieve faster convergence than ADMM, with a slightly better performance from our method. To illustrate the benefit of such acceleration, we take the results generated by each method within the same computational time, and use color-coding to visualize the edge-length error 
\begin{equation}
	\xi(e) = |e - l|/ l
	\label{eq:EdgeLengthError}
\end{equation}
where $e$ is the actual length for each edge. We can see that the two accelerated solvers lead to notably smaller edge-length errors than ADMM within the same computational time. The acceleration brings significant savings in computational time needed for a high-accuracy solution, which is required for the physical fabrication of the design~\cite{Grinspun14b}.

\begin{figure}[t!]
	\centering
	\includegraphics[width=\columnwidth]{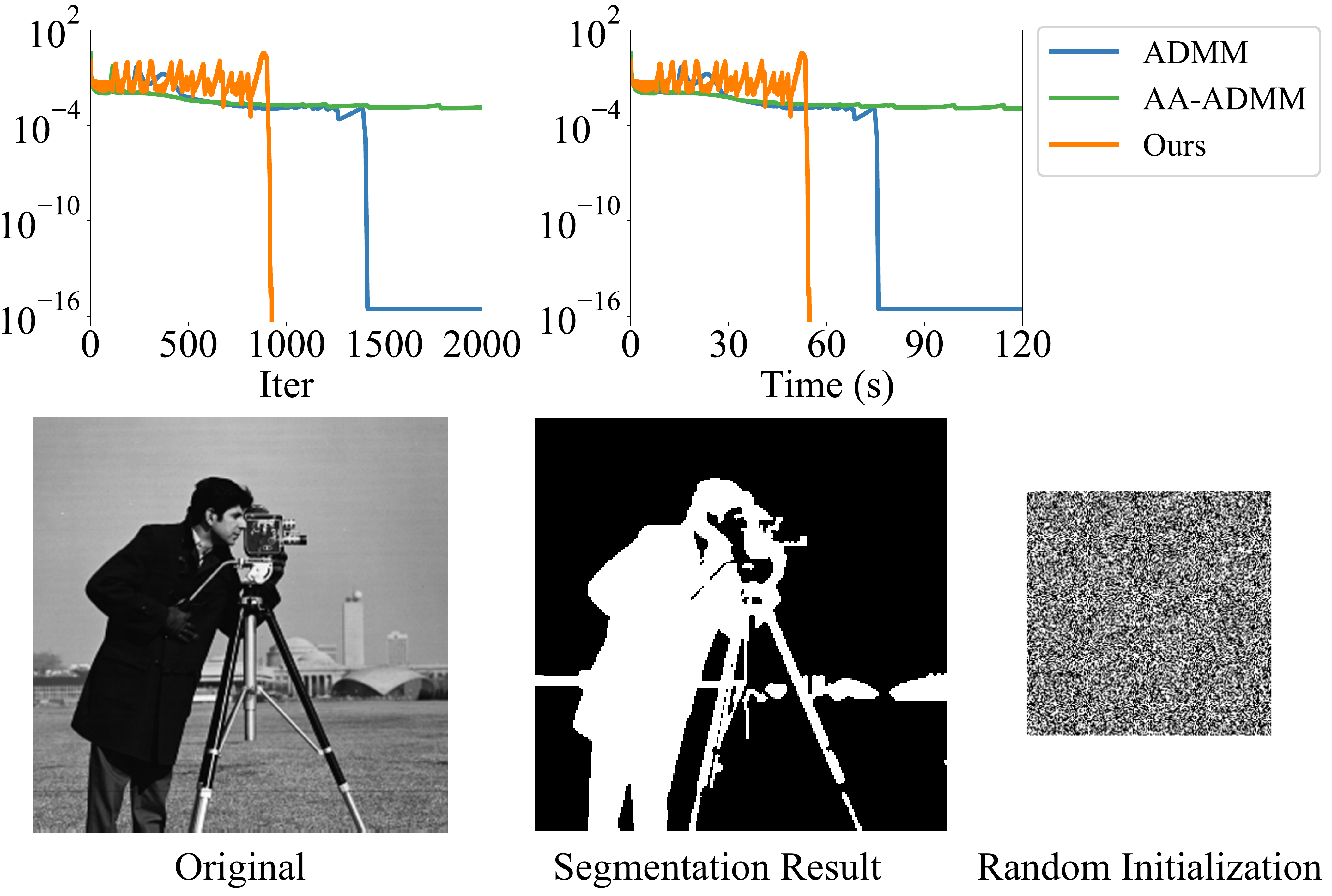}
	\caption{Comparison on the image segmentation problem~\eqref{eq:LpBoxADMM} with a re-formulated binary constraint. Our method reduces the iteration count and computational time required for convergence, while AA-ADMM fails to achieve acceleration.}
	\label{fig:LpBoxADMM}
\end{figure}

\para{Image Processing} In Fig.~\ref{fig:LpBoxADMM}, we test our method on the nonconvex ADMM solver for the following image segmentation problem~\cite{Wu2019-LpBox}:
\begin{equation}
	\min_{\bf x, \mathbf{z}}~\bf x^T \bf L \bf x + \bf{d}^T \bf x  + \iota_1(\bf z^1) + \iota_2(\bf z^2)\quad \textrm{s.t.}~\bf z^1=\bf x,~\bf z^2= \bf x, 
	\label{eq:LpBoxADMM}
\end{equation}
where $\mathbf{x} \in \mathbb{R}^n$ represents the pixel-wise labels to be optimized, $\bf L$ is a Laplacian matrix based on the similarity between adjacent pixels, $\bf d$ is a unary cost vector, and $\mathbf{z} = (\bf z^1, \bf z^2)$ is an auxiliary variable with $\bf z^1, \bf z^2 \in \mathbb{R}^n$.
$\iota_1$ and $\iota_2$ are indicator functions for the feasible sets $\mathcal{S}_1 = [0,1]^n$ and $\mathcal{S}_2=\{\bf p \in \mathbb{R}^n \mid \sum_{i=1}^n (p_i - \frac{1}{2})^2 = \frac{n}{4}\}$ respectively, which together with the linear constraint between $\mathbf{x}$ and $\mathbf{z}$ induces a binary constraint for the labels $\mathbf{x}$.
Fig.~\ref{fig:LpBoxADMM} uses the cameraman image to compare ADMM, AA-ADMM, and our method with the primal residual norm as the merit function, using the same random initialization. We use the python source code released by the authors\footnote{\url{https://github.com/wubaoyuan/Lpbox-ADMM}} for the ADMM implementation, and modify it to implement AA-ADMM and our method. We use the general approach of AA-ADMM since the problem does not meet the reduced-variable conditions.
The released code gradually changes the penalty parameter $\beta$, starting with $\beta = 5$ and increasing it by $3\%$ every five iterations until it reaches the upper bound $1000$. Since a different value of $\beta$ will lead to a different fixed-point iteration, for both AA-ADMM and our method we reset the history of Anderon acceleration when $\beta$ changes.
We observe an interesting behavior of the ADMM solver: initially it maintains a relatively high value of the combined residual norm until the variable $\mathbf{z}$ converges to its value $\mathbf{z}^\ast$ in the solution; afterwards, $\mathbf{z}$ remains close to $\mathbf{z}^\ast$, and the ADMM iteration effectively reduces to an affine transformation for the variables $\mathbf{x}$ and $\mathbf{y}$ with a rapid decrease of the combined residual norm. In comparison, our method shows more oscillation of the combined residual norm in the initial stage but accelerates the convergence of $\mathbf{z}$ towards $\mathbf{z}^\ast$, followed by a similar rapid decrease of the combined residual norm, thus outperforming ADMM in both iteration count and computational time. On the other hand, AA-ADMM fails to achieve acceleration. 

\begin{figure}[t!]
	\centering
	\includegraphics[width=\columnwidth]{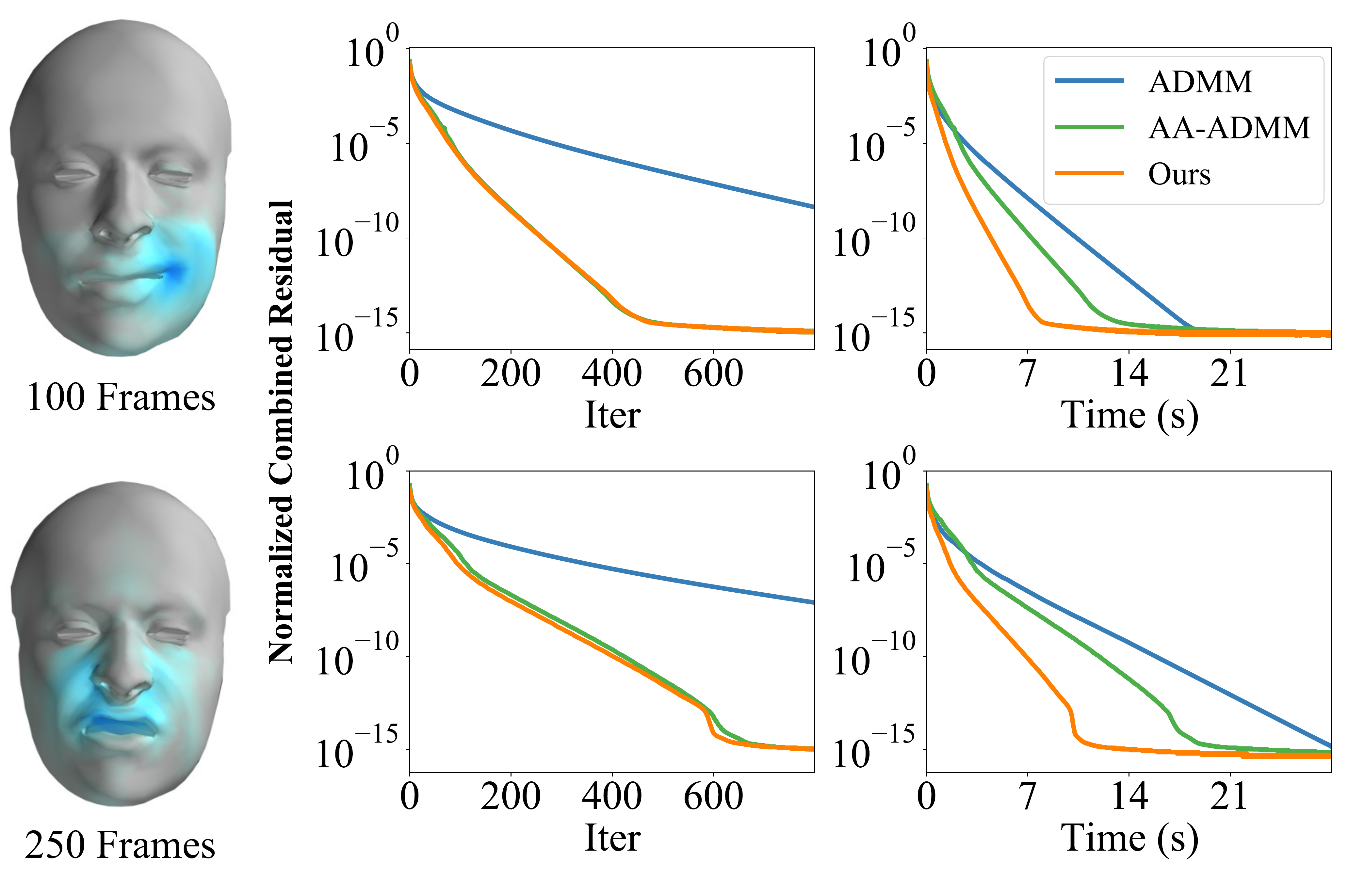}
	\caption{Comparison on a convex problem~\eqref{eq:splocs} with $\lambda = 2$, for computing local mesh deformation components from an input mesh sequence and given weights. The methods are tested using two mesh sequences constructed from the facial expression dataset of~\cite{COMA:ECCV18}, with 100 frames and 250 frames, respectively. We set the penalty parameter to $\beta=10$ for both problem instances. Our method have similar acceleration performance as AA-ADMM in reducing the number of iterations, and outperforms AA-ADMM in actual computational time.}
	\label{fig:Splocs}
\end{figure}

\begin{figure*}[t!]
	\centering
	\includegraphics[width=\textwidth]{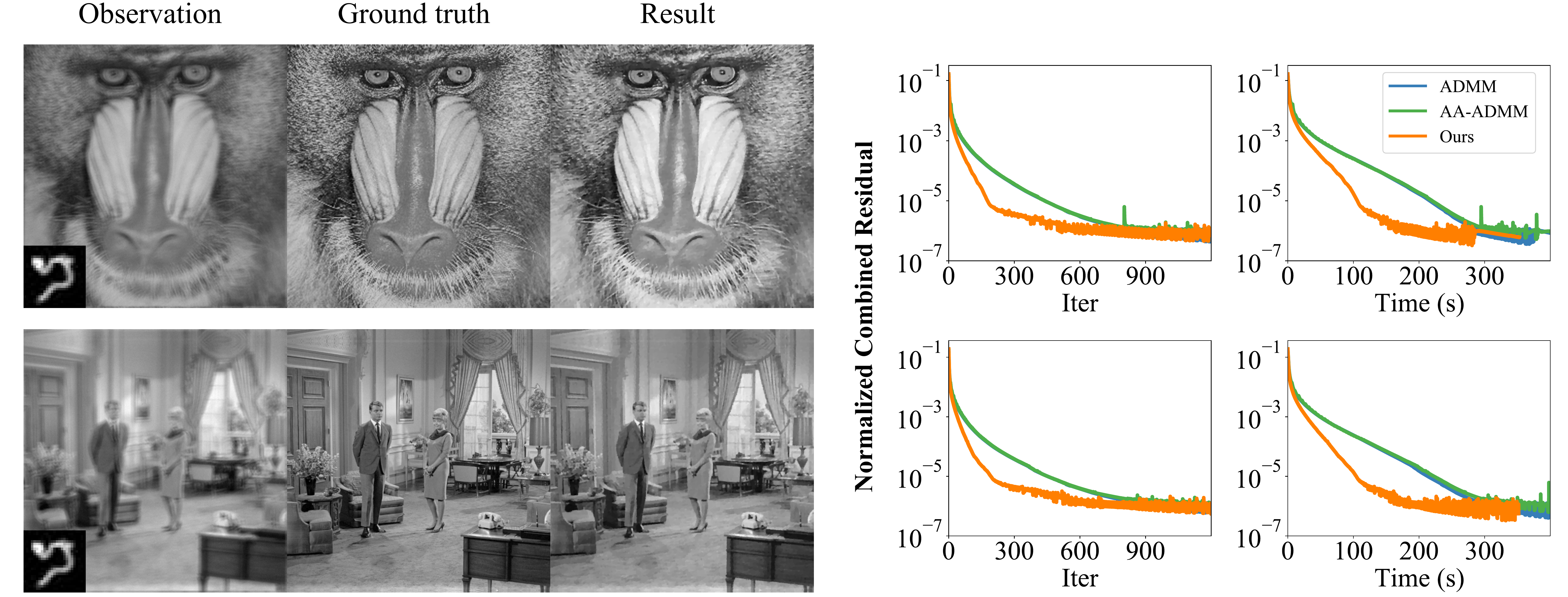}
	\caption{Comparison on the convex problem~\eqref{eq:ImageProcessing} for image deconvolution. We choose $\lambda = 400$ in problem~\eqref{eq:ImageProcessing}, and set the penalty parameter to $\beta = 100$. ADMM and AA-ADMM have fairly similar performance. Both are outperformed by our method.}
	\label{fig:ImageProcessing}
\end{figure*}

\para{Convex Problems} 
Although our method is designed with nonconvex problems in mind, it can be naturally applied to convex problems. 
In Fig.~\ref{fig:Splocs}, we apply our method to the ADMM solver in~\cite{Neumann2013} for computing mesh deformation components given a mesh animation sequence and component weights:
\begin{equation}
\argmin_{\mathbf{X}, \mathbf{Z}} \|\mathbf{V}-\mathbf{W} \mathbf{Z}\|_F^2 + \lambda \cdot \Omega_1 (\mathbf{X})\quad \text{s.t.}~\mathbf{X} = \mathbf{Z},
\label{eq:splocs}
\end{equation}
where matrix $\mathbf{Z}$ represents the deformation components to be optimized, $\mathbf{V}$ is the input mesh sequence, $\mathbf{W}$ represents the given weights for the components, $\mathbf{X}$ is an auxiliary variable, and $\Omega_1(\mathbf{X})$ is a weighted $\ell_1/\ell_{2}$-norm to induce local support for the deformation components. 
In Fig.~\ref{fig:ImageProcessing}, we accelerate the ADMM solver in~\cite{Heide2016} for image deconvolution:
\begin{equation}
\argmin_{\mathbf{x}, \mathbf{z}} \|\mathbf{x}_1-\mathbf{f}\|^2 + \lambda \cdot \Omega_2 (\mathbf{x}_2)\quad \text{s.t.}~\mathbf{K}\mathbf{z} = \mathbf{x}_1,~\mathbf{G}\mathbf{z} = \mathbf{x}_2,\\ 
\label{eq:ImageProcessing}
\end{equation}
where $\mathbf{z}$ represents the image to be recovered, $\mathbf{x} = (\mathbf{x}_1, \mathbf{x}_2)$ are auxiliary variables, matrix $\mathbf{K}$ represents the convolution operator, $\mathbf{G}$ is the image gradient matrix, and $\Omega_2$ is the $\ell_1/\ell_2$-norm for regularizing the image gradients.
Both problems~\eqref{eq:splocs} and~\eqref{eq:ImageProcessing} are convex, and AA-ADMM can only be applied using the general approach due to the problem structures. 
For both problems, we apply our method using the primal residual norm as the merit function.
We use the source codes released by the authors\footnote{\url{https://github.com/tneumann/splocs}}\footnote{\url{https://github.com/comp-imaging/ProxImaL}} to implement the ADMM solver and their accelerated versions. 
For both problems, our method accelerates the convergence of ADMM and outperforms AA-ADMM in the computational time.

\para{Limitation} Similar to~\cite{zhang2019accelerating}, our method may not be effective for ADMM solvers with very low computational cost per iteration. Fig.~\ref{fig:GeodesicOpt} shows the performance of our method and AA-ADMM on the ADMM solver from~\cite{Tao2019} for recovering a geodesic distance function on a mesh surface from a unit tangent vector field. The two methods achieve almost the same effect in reducing the amount of iterations required for convergence. Our method requires a shorter computational time than AA-ADMM to achieve the same value of combined residual, because we can only apply the general approach of AA-ADMM to this problem and its overhead is higher than our method. On the other hand, both approaches take a longer time than the original ADMM solver to achieve convergence, because the very low computational cost per iteration of the original solver means high relative overhead for both acceleration techniques.

\begin{figure}[t!]
	\centering
	\includegraphics[width=\columnwidth]{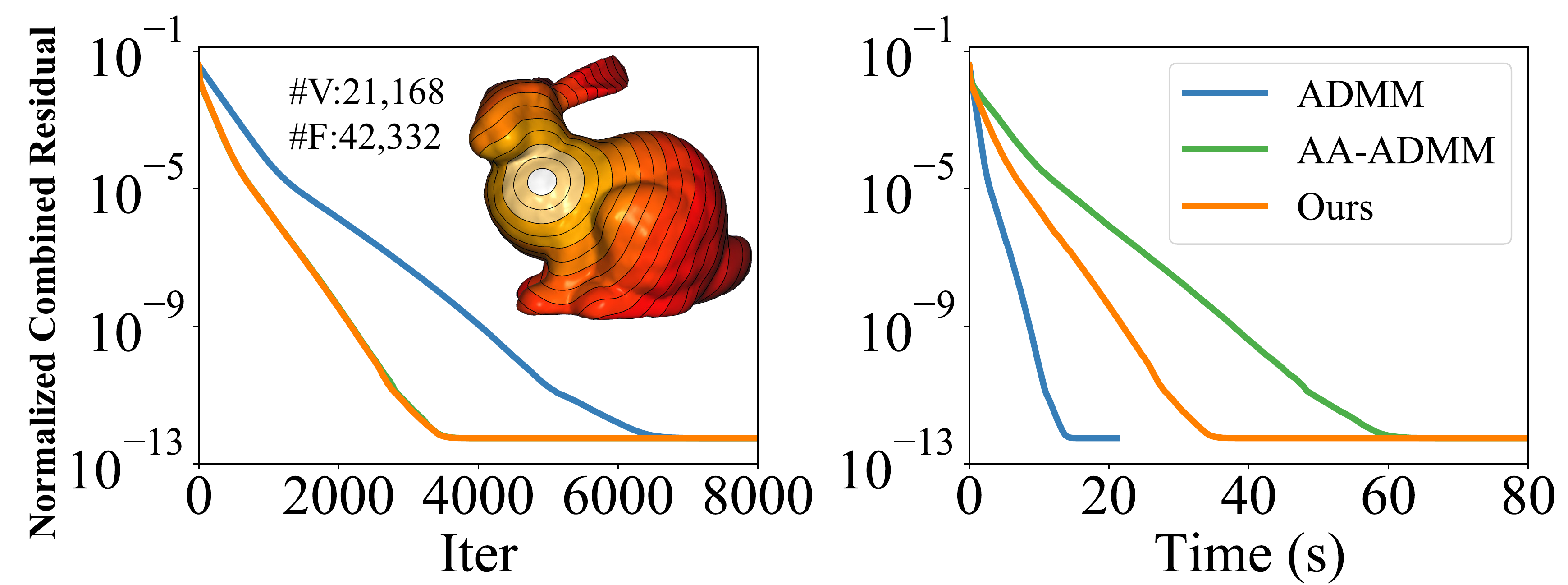}
	\caption{Comparison on the ADMM solver in~\cite{Tao2019} for recovering geodesic distance on meshes. Both AA-ADMM and our method can reduce the number of iterations required for convergence, but their actual computational time is higher due to the very low computational cost per iteration for the ADMM solver. Our method takes a shorter time than AA-ADMM thanks to its lower overhead.}
	\label{fig:GeodesicOpt}
\end{figure} 
	\section{Concluding Remarks}
In this paper, we propose an acceleration method for ADMM by applying Anderson Acceleration on its equivalent DR splitting formulation. Based on a fixed-point interpretation of DR splitting, we accelerate one of its variables that is not explicitly available in ADMM but can be derived from a linear transformation of the ADMM variables. 
Our strategy consistently outperforms the general Anderson acceleration approach in~\cite{zhang2019accelerating} due to the lower dimensionality of the accelerated variable. Compared to the reduced-variable approach in~\cite{zhang2019accelerating}, our method has the same dimensionality for the accelerated variable and achieves similar performance, but imposes no special requirements on the problem except for the separability of its target function. This makes our approach applicable to a much wider range of problems.
In addition, we analyze the convergence of the proposed algorithm, and show that it converges to a stationary point of the ADMM problem under appropriate assumptions. Various ADMM solvers in computer graphics and other domains have been tested to verify the effectiveness and efficiency of our algorithm.

There are still some limitations for our approach. First, the equivalence between ADMM and DR splitting relies on a separable target function for the ADMM problem. As a result, our method is not applicable to problems where the target function is not separable. However, as far as we are aware of, the majority of ADMM problems in computer graphics, computer vision, and image processing have a separable target function. 
Moreover, as shown in the geometry optimization example in Section~\ref{sec:Results}, it is possible to reformulate the problem to make the target function separable.
Therefore, this issue does not hinder the practical application of our method. 
Another limitation is that there is no theoretical guarantee that the method can always accelerate the convergence even locally. Recently, \cite{evans2020proof} provide theoretical results showing that Anderson Acceleration can improve the convergence rate, but their proofs require the original iteration to be contractive or converge q-linearly. For nonconvex DR splitting, to the best of our knowledge, local q-linear convergence can only be shown in very special cases that is too restrictive in practice. Further investigation of the theoretical property of Anderson Acceleration and nonconvex DR splitting is needed to provide a theoretical guarantee for acceleration.

	\paragraph*{Acknowledgements}
	The authors thank Andre Milzarek for proof-reading the paper and providing valuable comments.
	The target model in Figure~\ref{fig:WireMesh}, \href{https://www.thingiverse.com/thing:146386}{``Male Torso, Diadumenus Type''} by \href{https://www.thingiverse.com/CosmoWenman/about}{Cosmo Wenman}, is licensed under \href{https://creativecommons.org/licenses/by/3.0/}{CC BY 3.0}
	This research was partially supported by National Natural Science Foundation of China (No. 61672481), Youth Innovation Promotion Association CAS (No. 2018495), Zhejiang Lab (No. 2019NB0AB03). Wenqing Ouyang's work was partly supported by the Shenzhen Research Institute of Big Data (SRIBD). Yue Peng was supported by China Scholarship Council (No. 201906340085).
	
\bibliographystyle{eg-alpha-doi} 
\bibliography{reference} 
\appendix
\section{\textbf{Derivation of ADMM for Problem~\eqref{eq:GeometryOptProblem2}}}
\label{appx:GeometryOpt}
In this section, we derive an ADMM solver for the geometry optimization problem~\eqref{eq:GeometryOptProblem2} using the scheme~\eqref{eq:SeparableADMMX}--\eqref{eq:SeparableADMMZ}. We first write the problem in matrix form as
\begin{align*}
\min_{\mathbf{x}, \mathbf{z}} &~~\frac{1}{2}\left\|\mathbf{L}(\mathbf{x} - \tilde{\mathbf{x}})\right\|^2 + \sum_{i \in \mathcal{S}} \frac{w_i}{2} \left(D_{\mathcal{C}_i}(\mathbf{z}_i)\right)^2  + \sum_{j \in \mathcal{H}} \sigma_{\mathcal{C}_j} (\mathbf{z}_j)\\
\textrm{s.t.} &~~\mathbf{A}\mathbf{x} - \mathbf{z} = \mathbf{0},
\end{align*}
where matrix $\mathbf{A}$ stacks all matrices $\{\mathbf{A}_i \mid i \in \mathcal{S}\}$ and $\{\mathbf{A}_j \mid j \in \mathcal{H}\}$. In the following, $\mathbf{y}$ denotes the dual variable that consists of $\{\mathbf{y}_i \mid i \in \mathcal{S}\}$ and $\{\mathbf{y}_j \mid j \in \mathcal{H}\}$ corresponding to the soft constraints $\mathcal{S}$ and hard constraints $\mathcal{H}$, respectively.
We will use superscripts to indicate iteration counts, to avoid conflict with subscripts that indicate the constraints. Then the step~\eqref{eq:SeparableADMMX} reduces to the problem
\begin{equation}
	\min_{\mathbf{x}}~\frac{1}{2}\left\|\mathbf{L}(\mathbf{x} - \tilde{\mathbf{x}})\right\|^2 + \frac{\beta}{2}\|\mathbf{A}\bf{x}-\bf{z}^k+\bf{y}^k\|^2,
	\label{eq:GeomOptXProblem}
\end{equation}
which can be solved via the linear system
\begin{equation}
	(\mathbf{L}^T\mathbf{L} + \beta \mathbf{A}^T \mathbf{A}) \mathbf{x}^{k+1} = \mathbf{L}^T\mathbf{L}\tilde{\mathbf{x}} + {\beta}\mathbf{A}^T(\bf{z}^k-\bf{y}^k).
	\label{eq:GeomOptXSystem}
\end{equation}
The step~\eqref{eq:SeparableADMMY} is simply written as
\begin{equation}
	\bf{y}^{k+1} = \bf{y}^k+\bf{A} \bf{x}^{k+1} - \bf{z}^{k}.
	\label{eq:GeomOptYStep}
\end{equation}
The step~\eqref{eq:SeparableADMMZ} reduces to separable subproblems:
\begin{align}
\min_{\mathbf{z}_i} & ~~\frac{w_i}{2}\left(D_{\mathcal{C}_i}(\mathbf{z}_i)\right)^2 + \frac{\beta}{2}\|\bf{A}_i\bf{x}^{k+1}-\bf{z}_i+\bf{y}_i^{k+1}\|^2 \quad \textrm{for } i \in \mathcal{S},\label{eq:ZProblemSoft}\\
\min_{\mathbf{z}_j} & ~~\sigma_{\mathcal{C}_j} (\mathbf{z}_j) + \frac{\beta}{2}\|\bf{A}_j\bf{x}^{k+1}-\bf{z}_j+\bf{y}_j^{k+1}\|^2 \quad \textrm{for } j \in \mathcal{H}. \label{eq:ZProblemHard}
\end{align}
The solution to~\eqref{eq:ZProblemHard} is
\begin{equation}
	\mathbf{z}_j^{k+1} = P_{\mathcal{C}_j} ( \bf{A}_j\bf{x}^{k+1}+\bf{y}_j^{k+1} ),
	\label{eq:GeomOptZjUpdate}
\end{equation}
where $P_{\mathcal{C}_j}(\cdot)$ is a projection operator onto the $\mathcal{C}_j$. The solution to~\eqref{eq:ZProblemSoft} is
\begin{equation}
 	\mathbf{z}_i^{k+1} = \frac{w_i \cdot P_{\mathcal{C}_i} (\bf{A}_i\bf{x}^{k+1}+\bf{y}_i^{k+1}) + \beta \cdot (\bf{A}_i\bf{x}^{k+1}+\bf{y}_i^{k+1})}{w_i + \beta}.
 	\label{eq:GeomOptZiUpdate}
\end{equation}

\section{\textbf{Proof for Proposition~\ref{prop1-9}}}
\label{app-A}
\begin{proof}
By Proposition~\ref{thm:prop5.2} we have
$$ \bf u(\bf s)=\bf{A}\bf{\bar{x}}=\prox{\gamma\varphi_1}(\bf s).          $$
doing a simple change of variables, similar result for $\bf v(\bf s)$ can be attained as
$$ \bf v(\bf s)=\bf{B}\bf{\bar{z}}+\bf c=\prox{\gamma\varphi_2}(2\bf u(\bf s)-\bf s),   $$
For the expression of DR envelope, we utilize the optimality condition of $\prox{\gamma\varphi_1}$
\begin{align}
\label{eqA-1}
 \nabla\varphi_1(\bf u(\bf s))+\frac{1}{\gamma}(\bf u(\bf s)-\bf s)=0 \quad\Rightarrow\quad 2\bf u(\bf s)-\bf s=\bf u(\bf s)-\gamma\nabla\varphi_1(\bf u(\bf s)).
\end{align}
We then rewrite $\dre$ as
\begin{align*}
    \dre(\bf s)
    = \min\limits_{\bf w}&\left\{\varphi_1(\bf u(\bf s))+\varphi_2(\bf w)+\frac{1}{2\gamma}\|\bf w-(\bf u(\bf s)-\gamma\nabla\varphi_1(\bf u(\bf s)))\|^2\right.\\
    & \quad \left.-\frac{1}{2\gamma}\|\nabla\varphi_1(\bf u(\bf s))\|^2\right\}.    
\end{align*}
The definition of $\bf v(\bf s)$ indicates that $\bf v(\bf s)$ is the solution of minization problem in the definition of $\dre$, so
\begin{align*}
&\dre(\bf s)\\
  =~&\varphi_1(\bf u(\bf s))+\varphi_2(\bf v(\bf s))+\langle  \nabla\varphi_1(\bf u(\bf s)) ,\bf v(\bf s)-\bf u(\bf s) \rangle+\frac{1}{2\gamma} \|\bf v(\bf s)-\bf u(\bf s)\|^2  \\
  =~&\varphi_1(\bf u(\bf s))+\varphi_2(\bf v(\bf s))+\frac{1}{\gamma}\langle \bf s-\bf u(\bf s) ,\bf v(\bf s)-\bf u(\bf s) \rangle+\frac{1}{2\gamma} \|\bf v(\bf s)-\bf u(\bf s)\|^2.
\end{align*}
By Proposition~\ref{thm:prop5.2} we have
$$\varphi_1(\bf u(\bf s))=f(\bf{\bar x}),\quad \varphi_2(\bf v(\bf s))=g(\bf{ \bar z}),$$
which completes the proof.
\end{proof}

\section{\textbf{Proof for Proposition~\ref{prop:StationaryPointRecovery}}}
\label{app-B}
\begin{proof}
By the definition of $\bf s^*$, we know that for $\bf v^*=\bf B\bf z^*+\bf c$. The definition of fixed-point indicates $\bf v^*=\bf u^*$. Hence $\bf A\bf x^*-\bf B\bf z^*-\bf c=0$. We then utilize the definitions of $\bf x^*$ and $\bf z^*$ and the optimality conditions of the associated minimization problems:\begin{align*}
   -\frac{1}{\gamma}\bf A^T(\bf A\bf x^*-\bf s^*) &\in \partial f(\bf x^*),  \\
 -\frac{1}{\gamma}\bf B^T(\bf B\bf z^*+\bf c-(2\bf u^*-\bf s^*)) &\in \partial g(\bf z^*). 
\end{align*}
We note that $\frac{1}{\gamma}=\beta$, which means 
\begin{align*}
 -\beta \bf A^T\bf y^*= \frac{1}{\gamma}\bf A^T(\bf u^*-\bf s^*) &\in \partial f( \bf x^*),  \\
\beta \bf B^T\bf y^*= -\frac{1}{\gamma}\bf B^T(\bf s^*-\bf u^*) &\in \partial g(\bf z^*).
\end{align*}
This completes the proof.
\end{proof}

\section{\textbf{Proof for Theorem~\ref{thm2-9}}}
\label{app-C}
Let us first prove two lemmata:
\begin{lemma}\label{lemma2-4}
Assume that $\bf{s}^*$ is a fixed point of $\mathcal{G}$, $\varphi_1$ is differentiable, and $\prox{\gamma\varphi_1}$ is single-valued. Then $\bf{u}^*=\prox{\gamma\varphi_1}(\bf{s}^*)$ is a stationary point of \eqref{eq:DRProblem}.
\end{lemma}
\begin{proof}
By the definition of $\mathcal{G}$
$$  \bf{u}^*\in \prox{\gamma\varphi_2}(2\bf{u}^*-\bf{s}^*) . $$
By the optimality condition of $\prox{\gamma\varphi_2}$
$$  \frac{1}{\gamma}(\bf{u}^*-\bf{s}^*)\in\partial\varphi_2(\bf{u}^*).      $$
By the definition of $\bf{u}^*$ and the optimality condition of $\prox{\gamma\varphi_1}$
$$  0=\nabla\varphi_1(\bf{u}^*)+\frac{1}{\gamma}(\bf{u}^*-\bf{s}^*),      $$
which means
$ 0\in \nabla\varphi_1(\bf{u}^*)+\partial\varphi_2(\bf u^*) $.
\end{proof}
\begin{lemma}
\label{lemma2-6}
Let $\bf{s}^*$ and $\bf{u}^*$ be defined in Lemma~\ref{lemma2-4} and assume the conditions in Lemma \ref{lemma2-4} hold.
Moreover, define 
\[
\bf Z_{\beta}(\bf s)=\argmin\limits_{\bf z\in\mathbb{R}^n}\left(g(\bf z)+\frac{\beta}{2}\|\bf B\bf z+\bf c-\bf s\|^2\right).
\]
If $\bf{u}^*\in \bf B\bf Z_\beta(2\bf{u}^*-\bf{s}^*)+\bf c$, and $(\bf{x}^*,\bf{y}^*,\bf{z}^*)$ satisfies
\begin{align*}
\bf{x}^*&\in\argmin\limits_{\bf x}f(\bf x)+\frac{\beta}{2}\|\bf A\bf x-\bf{s}^*\|^2 , \\
\bf{y}^*&=\bf{u}^*-\bf{s}^*,    \\
\bf{u}^*&=\bf B\bf{z}^*+\bf c,\quad \bf{z}^*\in\bf Z_\beta(2\bf{u}^*-\bf{s}^*),
\end{align*}
then $(\bf{x}^*,\bf{y}^*,\bf{z}^*)$ is a stationary point of \eqref{eq:SeparableADMMProblem}.
\end{lemma}
\begin{proof}
By Proposition~\ref{thm:prop5.2} we have $\bf A\bf{x}^*\in\prox{\gamma\varphi_1}(\bf{s}^*)$. Since $\prox{\gamma\varphi_1}$ is single-valued we have $\bf{u}^*=\bf A\bf{x}^*$. By \eqref{eqA-1}
$$ \frac{1}{\gamma}(\bf{s}^*-\bf{u}^*)=\nabla\varphi_1(\bf{u}^*). $$
By \cite[Proposition 5.3]{themelis2020douglas} we have
$$ \bf A^T\nabla\varphi_1(\bf{u}^*)= \bf A^T\hat{\partial}\varphi_1(\bf{u}^*)\subset \hat{\partial} f(\bf{x}^*),   $$
where we have $\hat{\partial}\varphi_1(\bf{u}^*)=\{\nabla\varphi_1(\bf{u}^*)\}$ by \cite[Exercis 8.8]{rockafellar2009variational}. Notice that $\beta=\frac{1}{\gamma}$ we then have
$$  -\beta\bf A^T\bf{y}^* \in\hat{\partial}f(\bf{x}^*)\subset\partial f(\bf{x}^*).  $$
Similarly we have
$$ \beta \bf{y}^*=\frac{1}{\gamma}(\bf{u}^*-\bf{s}^*) \in\hat{\partial} \varphi_2(\bf{u}^*).       $$
By Proposition~\ref{thm:prop5.2} and  \cite[Exercise 8.8]{rockafellar2009variational}
$$   \beta \bf B^T\bf{y}^*\in \bf B^T\hat{\partial}\varphi_2(\bf{u}^*)\subset \hat{\partial} g(\bf{z}^*)\subset \partial g(\bf{z}^*).        $$
Finally, we have
\begin{equation*}
\bf A\bf{x}^*=\bf{u}^*=\bf B\bf{z}^*+\bf c\quad\Rightarrow\quad \bf A\bf{x}^*-\bf B\bf{z}^*-\bf c=0. \qedhere
\end{equation*}
\end{proof}
Finally, we give the main proof for Theorem~\ref{thm2-9}.
\begin{proof}
The proof here is similar to the proof for \cite[Theorem 4.1]{themelis2020douglas}. Let $\eta=\min\{\nu_1,\frac{c}{(1+\gamma L)^2}\}$, where $c$ is the constant defined in \cite[Theorem 4.1]{themelis2020douglas}, then by algorithmic construction we have
$$    \dre(\bf{s}_k)-\dre(\bf s_{k+1})\geq \eta\|\bf{s}_k-\bf G(\bf s_{k})\|^2=\eta\|\bf{v}_k-\bf{u}_k\|^2.         $$
Due to the definition of $\dre$ and the fact that $\varphi_1,\varphi_2$ are both proper, we have $\dre(\bf s_0)<\infty$. By Assumption (A.3) we know $\varphi=\varphi_1+\varphi_2$ is bounded from below and then by \cite[Proposition 3.4]{themelis2020douglas} $\dre$ is also bounded from below. Hence
$$  \eta\sum\limits_{k=0}^\infty \|\bf{u}_k-\bf{v}_k\|^2<\infty\quad \Rightarrow \quad \|\bf{u}_k-\bf{v}_k\|\rightarrow0,    $$
which proves (a). For (b) we first note that since $\gamma<\frac{1}{L}$, by \cite[Theorem 3.1]{themelis2020douglas} $\dre$ is level-bounded provided Assumption (B.2) holds. So by (a) we know $\{\bf s_k\}$ is bounded.  Then by \cite[Proposition 2.3]{themelis2020douglas} we know $\prox{\gamma\varphi_1}$ is Lipschitz continuous. Therefore $\{\bf u_k\}$ is also bounded. The boundedness of $\{\bf v_k\}$ follows from $\|\bf v_k-\bf u_k\|\rightarrow 0$. Therefore $\bf u_{k_i}\rightarrow \bf{u}^*$. Next, to prove that $\bf s^*$ is a fixed-point of $\mathcal G$, it suffices to show $\bf u^*\in\prox{\gamma\varphi_2}(2\bf u^*-\bf s^*)$. By the continuity of $\prox{\gamma\varphi_1}$ we know $\bf u_{k_i}\rightarrow \bf u^*$. Then since $\|\bf{u}_k-\bf{v}_k\|\rightarrow 0$, we also have $\bf v_{k_i}\rightarrow \bf{u}^*$. Notice that $\bf v_{k_i}\in\prox{\gamma\varphi_2}(2\bf u_{k_i}-\bf s_{k_i})$, by Assumption (B.2) and \cite[Theorem 1.25]{rockafellar2009variational} we know that $\prox{\gamma\varphi_2}$ is outer semicontinuous(osc), then by \cite[Exercise 5.30]{rockafellar2009variational}  
$$  \bf{u}^*=\lim\limits_{i\rightarrow\infty}\bf v_{k_i}\subset\limsup\limits_{i\rightarrow\infty}\prox{\gamma\varphi_2}(2\bf u_{k_i}-\bf s_{k_i})\subset \prox{\gamma\varphi_2}(2\bf{u}^*-\bf{s}^*) ,   $$
which proves that $\bf s^*$ is fixed-point of $\mathcal G$. The stationarity of $\bf u^*$ follows from Lemma~\ref{lemma2-4}.  

For (c), notice that if Assumption (B.3) holds, then $\bf Z_\beta$ is locally bounded and osc by \cite[Theorem 1.17]{rockafellar2009variational}. Since $\bf z_{k_i}\in \bf Z_\gamma(2\bf u_{k_i}-\bf s_{k_i})$ and $2\bf u_{k_i}-\bf s_{k_i}\to 2\bf{u}^*-\bf{s}^*$, $\{\bf z_{k_i}\}$ is bounded.  So the cluster point of $\{\bf z_{k_i}\}$, $\bf{z}^*$ must exists. Without loss of generality, we can assume $\bf z_{k_i}\to \bf{z}^*$. By \cite[Exercise 5.30]{rockafellar2009variational}  
$$   \bf{z}^*=\lim\limits_{i\rightarrow\infty}\bf z_{k_i}\subset\limsup\limits_{i\rightarrow\infty}\bf Z_\beta(2\bf u_{k_i}-\bf s_{k_i})\subset  \bf Z_\beta(2\bf{u}^*-\bf{s}^*)+\bf c.      $$
Next, push to the limit on both sides of $\bf v_{k_i}=\bf B\bf z_{k_i}+\bf c$
$$   \bf{v}^*=\lim\limits_{i\to \infty}\bf v_{k_i}=  \lim\limits_{i\to \infty}\bf B\bf z_{k_i}+\bf c=\bf B\bf{z}^*+\bf c  . $$
The stationarity of $(\bf{x}^*,\bf{y}^*,\bf{z}^*)$ follows from Lemma~\ref{lemma2-6}.
\end{proof}

\section{\textbf{Further Discussion for Generating the Stationary Point of ADMM}}
\label{app-E}
For the most general case where $\prox{\gamma\varphi_1}$ and $\prox{\gamma\varphi_2}$ are both set-valued, we need much more sophisticated techniques to generate the stationary point of ADMM. 

We note the definition of $\bf s^*\in\mathcal G(\bf s^*)$ means there exist $\bf u^*$ such that $\bf u^*\in\prox{\gamma\varphi_1}(\bf s^*)$ and $\bf u^*\in\prox{\gamma\varphi_2}(2\bf u^*-\bf s^*)$. We assume $\bf u^*$ is known since $\bf u^*$ is explicitly available from Algorithm~\ref{algo1}. This is because proximal mapping is outer semi-continuous, which means if a subsequence $\bf s_{k_i}\rightarrow \bf s^*$, then it suffice to choose a cluster point of $\{\bf u_{k_i}\}$ to generate such a $\bf u^*$. Our goal is to generate the stationary point of ADMM from $(\bf s^*,\bf u^*)$.

We first need a technical lemma:
\begin{lemma}
\label{lemmaF-1}
Let $h:\mathbb{R}^n\rightarrow\bar{\mathbb{R}}$ and $\mathbf{C}\in\mathbb{R}^{p\times n}$. Suppose for some $\beta$ the set-valued mapping $\bf{X}_\beta(\bf{s}):=\argmin\limits_{\bf{x}\in\mathbb{R}^n}\{h(\bf{x})+\frac{\beta}{2}\|\bf{C}\bf{x}-\bf{s}\|^2 \}$ is nonempty for any $\bf s\in\mathbb{R}^p$.  Let $\gamma=1/\beta$ and $\varphi=h_{\bf C}$. If $\bf u^*\in\prox{\gamma\varphi}(\bf s^*)$ and 
$$   \bf x^*\in \argmin\limits_{\bf x\in\mathbb{R}^n}\{ h(\bf x)+\frac{\beta}{2}\|\bf C\bf x-\bf s^*\|^2+\frac{\alpha}{2}\|\bf C\bf x-\bf u^*\|^2  \},              $$
where $\alpha>0$. Then $\bf x^*\in\bf{X}_{\beta}(\bf s^*)$ and $\bf u^*=\bf C\bf x^*$.
\end{lemma}
\begin{proof}
We first prove $\bf C\bf x^*=\bf u^*$. Let $\bf u^+=\bf C\bf x^*$. By the definition of $\bf x^*$ we know $\forall \bf e \in \text{ker}(\bf C)$, we have $h(\bf x^*)\leq h(\bf x^*+\bf e)$, which means $\varphi(\bf u^+)=h(\bf x^*)$. Then we have:
\begin{align*}  &\varphi(\bf{u}^+)+\frac{\beta}{2}\|\bf{u}^+-\bf{s}^*\|^2+\frac{\alpha}{2}\|\bf u^+-\bf u^*\|^2\\&=h(\bf{x}^*)+\frac{\beta}{2}\|\bf C\bf{x}^*-\bf s^*\|^2+\frac{\alpha}{2}\|\bf C\bf x^*-\bf u^*\|^2 \\
&=\inf\limits_{\bf u} \{ \inf\limits_{\bf x:\bf C\bf x=\bf u}h(\bf x)+\frac{\beta}{2}\|\bf{u}-\bf{s}^*\|^2+\frac{\alpha}{2}\|\bf u-\bf u^*\|^2\}  \\
&=\inf\limits_{\bf u}\{ \varphi(\bf u)+\frac{\beta}{2}\|\bf{u}-\bf{s}^*\|^2+\frac{\alpha}{2}\|\bf u-\bf u^*\|^2 \}.  
\end{align*}
But by the definition of $\bf u^*$, we know
$$ \{\bf u^*\}=\argmin\limits_{\bf u}   \{ \varphi(\bf u)+\frac{\beta}{2}\|\bf{u}-\bf{s}^*\|^2+\frac{\alpha}{2}\|\bf u-\bf u^*\|^2 \} ,                             $$
which means $\bf u^+=\bf u^*$ and hence demonstrates that $\bf u^*=\bf C\bf x^*$. Now assume that $\bf x^*\notin \bf X_{\beta}(\bf s^*)$, which means there exists another $\bf x^+$ such that 
$$    h(\bf x^+)+\frac{\beta}{2}\|\bf C\bf x^+-\bf s^*\|^2< h(\bf x^*)+\frac{\beta}{2}\|\bf C\bf x^*-\bf s^*\|^2 .                          $$
However, we have \begin{align*}
 h(\bf x^*)+\frac{\beta}{2}\|\bf C\bf x^*-\bf s^*\|^2&=\varphi(\bf u^*)+\frac{\beta}{2}\| \bf u^*-\bf s^* \|^2             ,     \\
&\leq  \varphi(\bf C\bf x^+)+\frac{\beta}{2}\| \bf C\bf x^+-\bf s^* \|^2,  \\
&\leq  h(\bf x^+)+\frac{\beta}{2}\|\bf C\bf x^+-\bf s^*\|^2,
\end{align*}
which yields contradiction. Hence $\bf x^*\in\bf X_\beta(\bf s^*)$. 
\end{proof}
Then we are able to prove the general transition theorem:
\begin{theorem}
Suppose $\bf s^*$ is the fixed-point of $\mathcal G$, and $\bf u^*\in\prox{\gamma\varphi_1}(\bf s^*)\cap \prox{\gamma\varphi_2}(2\bf u^*-\bf s^*)$. Define:\begin{align*}
&\bf x^*\in\argmin\limits_{\bf x}f(\bf x)+\frac{1}{2\gamma}\|\bf A\bf x-\bf s^*\|^2+\frac{\alpha}{2}\|\bf A\bf x-\bf u^*\|^2, \\
&\bf y^*=\bf s^*-\bf u^* \\ &\bf z^*\in\argmin\limits_{\bf z}g(\bf z)+\frac{1}{2\gamma}\|\bf B\bf z+\bf c-(2\bf u^*-\bf s^*)  \|^2+\frac{\alpha}{2}\|\bf B\bf z+\bf c-\bf u^*\|^2,
\end{align*}
where $\alpha>0$, then $(\bf{x}^*,\bf{y}^*,\bf{z}^*)$ is a stationary point of~\eqref{eq:SeparableADMMProblem}.
\end{theorem}
\begin{proof}
Lemma~\ref{lemmaF-1} means that \begin{align*}
  &\bf u^*=\bf A\bf x^*,\bf u^*=\bf B\bf z^*+\bf c ,           \\
  &\bf x^*\in\argmin\limits_{\bf x}f(\bf x)+\frac{1}{2\gamma}\|\bf A\bf x-\bf s^*\|^2 ,\\
  &\bf z^*\in\argmin\limits_{\bf z}g(\bf z)+\frac{1}{2\gamma}\|\bf B\bf z+\bf c-(2\bf u^*-\bf s^*)  \|^2. 
\end{align*}
Next, we utilize the optimality conditions of the proximity operator $\prox{\gamma\varphi_1}(\bf s^*)$  
$$  \frac{1}{\gamma}(\bf s^*-\bf u^*)\in\hat{\partial}\varphi_1(\bf u^*).  $$
Then by \cite[Proposition 5.3]{themelis2020douglas}
$$  -\beta \bf A^T \bf y^*= \frac{1}{\gamma}\bf A^T(\bf s^*-\bf u^*)\in \hat{\partial}f(\bf x^*)\in \partial f(\bf x^*).  $$
Similarly we have 
$$  \beta\bf B^T\bf y^*\in \partial g(\bf z^*).   $$
The last condition follows from $\bf u^*=\bf A\bf x^*=\bf B\bf z^*+\bf c$.
\end{proof}

\section{\textbf{Proof for Theorem~\ref{thm4-4} and Remark \ref{remark4-6}}}
\label{app-F}
Our proof will utilize the Kurdyka-{\L}ojasiewicz (KL) inequality~\cite{AttBolSva13}. We will introduce several notations for the definition of KL property. Let $\mathcal{C}_{\eta}$ be the set consisting of all the concave and continuous function $\rho:[0,\eta)\rightarrow\mathbb{R}_+$ satisfying that
$$\rho\in C^1((0,\eta)),\quad \rho(0)=0, \quad \rho'(x)>0,\forall x\in(0,\eta).$$
We also consider a subclass of $\mathcal{C}_\eta$, called {{\L}ojasiewicz functions}
$$ \mathcal L := \{ \rho : \mathbb{R}_+ \to \mathbb{R}_+ ,\exists~m > 0, \, \theta \in [0,1): \rho(x) = q x^{1-\theta} \}. $$
Next we give the definition of the KL property:
\begin{defn} Let $\psi$ be a proper, lower semicontinuous function. We say that $\psi$ has the KL property at $\bf{\bar x} \in \dom{\partial\psi}$ if there exists $\eta \in (0,\infty]$, a neighborhood $U$ of $\bf{\bar x}$, and a function $\rho \in \mathcal{C}_\eta$ such that for all $\bf x \in U \cap \{\bf x \in \Rnn : 0 < \psi(\bf x)-\psi(\bf{\bar x}) < \eta\}$ the KL-inequality holds, i.e.,
\begin{align}
\label{eq:kl-ineq}
\rho^\prime(\psi(\bf x) - \psi(\bf{\bar x})) \cdot \dist(0,\partial \psi(\bf x)) \geq 1.
\end{align}
If the mapping $\rho$ can be chosen from $\mathcal L$ and satisfies $\rho(x) = q x^{1-\theta}$ for some $q > 0$ and $\theta \in [0,1)$, then we say that $\psi$ has the KL-property at $\bar x$ with exponent $\theta$.
%
\end{defn}
It is known that a variety of functions, which contains the sub-analytic function \cite{AttBolSva13}, have the KL property. So we will directly work with the KL property.

We note that by the definition of $\Dg$ it is clear that $\Dg(\bf{s}_k,\bf{u}_k,\bf{v}_k)=\dre(\bf{s}_k)$.

Assume $\{\bf{s}_k,\bf{u}_k,\bf{v}_k\}$ is generated by Algorithm~\ref{algo1}, then we define $\mathcal U$ to be the set consisting of all the cluster points of $\{\bf{s}_k,\bf{u}_k,\bf{v}_k\}$. Several structural properties of U are listed in the next proposition.
\begin{proposition}
\label{prop2-11}
Assume $\{\bf{s}_k,\bf{u}_k,\bf{v}_k\}$ is bounded. Then
\begin{itemize}
\item[(a)] $\mathcal U$ is nonempty and compact.
\item[(b)] $\dist((\bf{s}_k,\bf{u}_k,\bf{v}_k),\mathcal U)\rightarrow 0$.
\item[(c)] If the assumptions in Theorem \ref{thm2-9} hold, then $\Dg$ is constant and finite on $\mathcal U$.
\end{itemize}
\end{proposition}
\begin{proof}
For statement (a) and (b), see  \cite[Lemma 5(iii)]{BolSabTeb14}. For (c), by Theorem \ref{thm2-9} we can assume $\dre(\bf{s}_k)\rightarrow l^*$ where $l^*$ is finite. Now assume $(\bf{s}^*,\bf{u}^*,\bf{v}^*)\in\mathcal U$, then the proof in Theorem \ref{thm2-9} has already shown that $\bf{v}^*=\bf{u}^*\in\prox{\gamma\varphi_2}(2\bf{u}^*-\bf{s}^*)$. And we clearly have $\bf{u}^*=\prox{\gamma\varphi_1}(\bf{s}^*)$ by the continuity of $\prox{\gamma\varphi_1}$. Hence $\Dg(\bf{s}^*,\bf{u}^*,\bf{v}^*)=\dre(\bf{s}^*)$. Notice that $\dre$ is strictly continuous \cite[Proposition 3.2]{themelis2020douglas}, so we have $\dre(\bf{s}^*)=l^*$, which completes the proof.
\end{proof}

In the following we provide the main proof of global and r-linear convergence stated in Theorem~\ref{thm4-4} and Remark \ref{remark4-6}.

 \begin{proof}
These two conclusions trivially hold if Algorithm~\ref{algo1} terminates after finite steps, so in the rest of the proof we assume Algorithm \ref{algo1} generates infintely many steps. Let $\delta,\eta$ be the constants appearing in the definition of KL property.  Choose $k'$ sufficiently large such that for any $k\geq k'$ we have $$ \dist((\bf{s}_k,\bf{u}_k,\bf{v}_k),\mathcal U)<\delta,\quad 0<\Dg(\bf{s}_k,\bf{u}_k,\bf{v}_k)-\Dg(\bf{\bar{s}},\bf{\bar{u}},\bf{\bar{v}})  <\eta     $$
where $(\bf{\bar{s}},\bf{\bar{u}},\bf{\bar{v}})\in\mathcal U$. Such a $k'$ exists due to Proposition~\ref{prop2-11}.  Define $\delta_k=\rho(\Dg(\bf{s}_k,\bf{u}_k,\bf{v}_k)-\Dg(\bf{\bar{s}},\bf{\bar{u}},\bf{\bar{v}}))$. For $k\geq k'$ we utilize the concavity of $\rho$\begin{align}
\label{eq11}
   \delta_{k}-\delta_{k+1}&\geq \rho'(\delta_k)(\Dg(\bf{s}_k,\bf{u}_k,\bf{v}_k)-\Dg(\bf s_{k+1},\bf u_{k+1},\bf v_{k+1}))\notag\\&\geq \frac{\Dg(\bf{s}_k,\bf{u}_k,\bf{v}_k)-\Dg(\bf s_{k+1},\bf u_{k+1},\bf v_{k+1}) }{\dist(0,\partial\Dg(\bf{s}_k,\bf{u}_k,\bf{v}_k)) }.
\end{align}

Next, we estimate $\dist(0,\partial\Dg(\bf{s}_k,\bf{u}_k,\bf{v}_k))$:\begin{align*}
\nabla_s\Dg(\bf{s}_k,\bf{u}_k,\bf{v}_k)&=\frac{1}{\gamma}(\bf{v}_k-\bf{u}_k),\\
\nabla_u\Dg(\bf{s}_k,\bf{u}_k,\bf{v}_k)&=\nabla\varphi_1(\bf{u}_k)-\frac{1}{\gamma}(\bf{s}_k-\bf{u}_k)-\frac{1}{\gamma}(\bf{v}_k-\bf{u}_k)+\frac{1}{\gamma}(\bf{u}_k-\bf{v}_k),\\
&=\frac{2}{\gamma}(\bf{u}_k-\bf{v}_k)\\
\partial_v\Dg(\bf{s}_k,\bf{u}_k,\bf{v}_k)&=\partial\varphi_2(\bf{v}_k)+\frac{1}{\gamma}(\bf{s}_k-\bf{u}_k)+ \frac{1}{\gamma}(\bf{v}_k-\bf{u}_k)\ni0,
\end{align*}
where we have used \eqref{eqA-1} for the second equality and the optimality condition of $\prox{\gamma\varphi_2}$ for the third equality. These means \begin{align}
\label{eq12}
 \dist(0,\partial\Dg(\bf{s}_k,\bf{u}_k,\bf{v}_k))\leq \frac{\sqrt{5}}{\gamma}\|\bf{v}_k-\bf{u}_k\|.
\end{align}
We now consider the next two cases\begin{description}
\item[Case 1:]  $\bf s_{k+1}=\bf s^{AA}_k$ then\begin{align*}
&\Dg(\bf{s}_k,\bf{u}_k,\bf{v}_k)-\Dg(\bf s_{k+1},\bf u_{k+1},\bf v_{k+1})=\dre(\bf s_{k})-\dre(\bf s_{k+1})\\&\geq\nu_1\|\bf{v}_k-\bf{u}_k\|^2+\nu_2\|\bf s_{k+1}-\bf{s}_k\|^2.
\end{align*}

Moreover, by Young's inequality\begin{align*}
2\sqrt{\nu_1\nu_2} \|\bf s_{k+1}-\bf s_{k}\|&=2\frac{\sqrt{\nu_2}\|\bf s_{k+1}-\bf s_{k}\|}{\sqrt{ \| \bf{v}_k-\bf{u}_k \|}}\sqrt{\nu_1}\sqrt{\|\bf{v}_k-\bf{u}_k\|}\\&\leq \frac{\nu_2\|\bf s_{k+1}-\bf{s}_k\|^2}{\|\bf{v}_k-\bf{u}_k\|}+\nu_1\|\bf{v}_k-\bf{u}_k\|.
\end{align*}

    Then by \eqref{eq11} and \eqref{eq12}\begin{align*}
 \delta_k-\delta_{k+1}&\geq \frac{\gamma}{\sqrt5}(\frac{\nu_2\|\bf s_{k+1}-\bf{s}_k\|^2}{\|\bf{v}_k-\bf{u}_k\|}+\nu_1\|\bf{v}_k-\bf{u}_k\|)\\&\geq \frac{2\gamma\sqrt{\nu_1\nu_2}}{\sqrt5}\|\bf s_{k+1}-\bf{s}_k\| .
    \end{align*}

\item[Case 2:] $\bf s_{k+1}=\bf G(\bf{s}_k)$. Then $\bf s_{k+1}-\bf{s}_k=\bf{v}_k-\bf{u}_k$ and by \cite[Theorem 4.1]{themelis2020douglas}\begin{align*}
\Dg(\bf{s}_k,\bf{u}_k,\bf{v}_k)-\Dg(\bf s_{k+1},\bf u_{k+1},\bf v_{k+1})&=\dre(\bf s_{k})-\dre(\bf s_{k+1})\\
&\geq\frac{c}{(1+\gamma L)^2}\|\bf{v}_k-\bf{u}_k\|^2.
\end{align*}

where $c$ is the constant defined in \cite[Theorem 4.1]{themelis2020douglas}. By \eqref{eq11} and \eqref{eq12}
$$  \delta_k-\delta_{k+1}\geq \frac{c\gamma}{\sqrt{5}(1+\gamma L)^2}\|\bf s_{k+1}-\bf{s}_k\|.$$
\end{description}
Let $\bar{a}=\min\{ \frac{2\gamma\sqrt{\nu_1\nu_2}}{\sqrt5}    ,\frac{c\gamma}{\sqrt{5}(1+\gamma L)^2}\}$, then we have
\begin{align}
\label{eq13}
\delta_{k}-\delta_{k+1}\geq \bar{a}\|\bf s_{k+1}-\bf{s}_k\|.
\end{align}
Notice that $\delta_k$ is positive and monotone decreasing, summing \eqref{eq13} from $k'$ to $\infty$
\begin{align}
\label{eq14}
 \delta_{k'}\geq \bar{a}\sum\limits_{k=k'}^{\infty}\|\bf s_{k+1}-\bf{s}_k\|.
\end{align}

which means that $\{\bf{s}_k\}$ is a Cauchy sequence and hence converges to some point $\bf{s}^*$. By the continuity of $\prox{\gamma\varphi_1}$ we know $\bf{u}_k\rightarrow \bf{u}^*=\prox{\gamma\varphi_1}(\bf{s}^*)$. $\bf{v}^*=\bf{u}^*$ follows from $\|\bf{v}_k-\bf{u}_k\|\rightarrow 0$. $\bf{u}^*$ is fixed-point of $\mathcal{G}$ follows from Theorem~\ref{thm2-9}. For Remark~\ref{remark4-6}, we need the condition that $\Dg$ has the KL property at $\mathcal U$ with exponent $\theta\in(0,\frac12]$. Now assume $\rho(x)=qx^{1-\theta}$. By the definition of KL property \eqref{eq:kl-ineq} and \eqref{eq12}\begin{align*}
\frac{\sqrt{5}q(1-\theta)}{\gamma}\|\bf{v}_k-\bf{u}_k\| & \geq q(1-\theta) \dist(0,\partial \Dg(\bf{s}_k,\bf{u}_k,\bf{v}_k))\\&\geq (\Dg(\bf{s}_k,\bf{u}_k,\bf{v}_k)-\Dg(\bf{\bar{s}},\bf{\bar{u}},\bf{\bar{v}}))^{\theta}.       
\end{align*}

Hence we have
$$  \delta_k=q(\Dg(\bf{s}_k,\bf{u}_k,\bf{v}_k)-\Dg(\bf{\bar{s}},\bf{\bar{u}},\bf{\bar{v}}))^{1-\theta}\leq q( \frac{\sqrt{5}q(1-\theta)}{\gamma}\|\bf{v}_k-\bf{u}_k\|  )^{\frac{1-\theta}{\theta}}. $$
By elementary calculus, one can show that $( \frac{\sqrt{5}q(1-\theta)}{\gamma}\|\bf{v}_k-\bf{u}_k\|  )^{\frac{1-\theta}{\theta}} $ is monotone increasing on $\theta\in(0,\frac12]$ provided that $\frac{\sqrt{5}q\|\bf{v}_k-\bf{u}_k\|}{\gamma}<1$. Since $\|\bf{v}_k-\bf{u}_k\|\to 0$, we can assume that $k'$ is sufficiently large such that for any $k\geq k'$ we have $\frac{\sqrt{5}q\|\bf{v}_k-\bf{u}_k\|}{\gamma}<1$. Then
$$  \delta_k\leq   a_1\|\bf{v}_k-\bf{u}_k\|        .$$
where $a_1$ is some constant. Similar to the previous proof, we can show\begin{align}\label{eq15}
 \delta_k-\delta_{k+1}\geq a_2\|\bf{v}_k-\bf{u}_k\|.
\end{align}
where $a_2=\frac{\sqrt{5}}{\gamma}\min\{\nu_1,\frac{c}{(1+\gamma L)^2}\}$. Summing \eqref{eq15} from $k'$ to $\infty$

$$   \delta_{k'}\geq a_2\sum \limits_{k=k'}^\infty\|\bf{v}_k-\bf{u}_k\|.      $$
Therefore
$$    a_1\|\bf v_{k'}-\bf u_{k'}\|\geq a_2\sum\limits_{k=k'}^\infty\|\bf{v}_k-\bf{u}_k\|.       $$
Define $H_k=\sum\limits_{i=k}^{\infty}\|\bf v_i-\bf u_i\|$ we have
$$   a_1(H_{k'}-H_{k'+1})\geq a_2H_{k'}  \quad\Rightarrow\quad H_{k'+1}\leq \frac{a_1-a_2}{a_1}H_{k'}.   $$
Similarly, we can show for any $l\geq k'$ we have
$$ H_{l+1}\leq \frac{a_1-a_2}{a_1}H_l,$$
which means that $\{H_k\}$ converges q-linearly and since $H_{k}\geq \|\bf{v}_k-\bf{u}_k\|$ we get the r-linear convergence of $\|\bf{v}_k-\bf{u}_k\|$. Then
$$    a_1\|\bf{v}_k-\bf{u}_k\|\geq \delta_k\geq \bar{a}\sum\limits_{i=k}^\infty \|\bf s_{i+1}-\bf s_i\|\geq \|\bf s_{k+1}-\bf{s}_k\|,    $$
which proves the r-linear convergence of $\{\|\bf s_{k+1}-\bf{s}_k\|\}$ and further implies the r-linear convergence of $\{\bf{s}_k\}$.
\end{proof}

\section{\textbf{Proof for Proposition~\ref{prop4-7}}}
\label{app-G}
\begin{proof}
The properness of $\varphi_2$ are given in Proposition~\ref{thm:prop5.2}(i) since we assume all the ADMM subproblems has solution. The lower semicontinuity of $\varphi_2$ are given by \cite[Proposition 5.10]{themelis2020douglas} by assuming (D.4). Hence assumption (A.2) is satisfied. Assumption (A.1) comes from \cite[Theorem 5.13]{themelis2020douglas} by assuming (D.2) and (D.3). For Assumption (B.2), without loss of generality, we can assume $f$ is bounded from below and $g$ is level-bounded. Then it is clear that $\varphi_1$ is bounded from below and $\varphi_2$ is level-bounded, and hence $\varphi_1+\varphi_2$ is level-bounded.
\end{proof}

\section{\textbf{Verification of Assumptions for $\ell_q$ Regularized Logistic Regression Problem}}
\label{appx:LqLogReg}
In this section we will verify Assumption~(A.1) to (A.3), (B.2) and (B.4), (C.2) for the $\ell_q$ regularized logistic regression problem. In this problem, we have
 $$   f(\bf x )=\sum\limits_{i=1}^p\log(1+\exp(-b_i(\bf a_i^T\bf w+v))) ,g(\bf z)=p\cdot\lambda\cdot\sum\limits_{i=1}^{n}|z_i|^q      $$
where $\bf x=(\bf w,v)\in\mathbb{R}^{n+1}$. For this problem matrices $\bf A$ and $\bf B$ are all identity, so image functions have rather simple form, i.e., $\varphi_1=f$ and $\varphi_2=g$. It is well known that $f$ is Lipschitz differentiable, so (A.1) is satisfied. Moreover, $g$ is continuous and hence lower semicontinuous, so (A.2) is satisfied.  To prove (A.3), since $f+g$ is continuous, it suffices to show (B.2) hold, because (A.1) then follows by \cite[Theorem 1.9]{rockafellar2009variational}. For the level-boundedness of $f+g$, we have:
 \begin{proposition}
 Assume $b_i$ are not all $1$ or $-1$, then $f+g$ is level-bounded.
\end{proposition}
\begin{proof}
Without loss of generality, we can assume $b_1=1$ and $b_2=-1$. Let $\alpha\in\mathbb{R}$, and $S=\{\bf x\in\mathbb{R}^{n+1}: f(\bf x)+g(\bf x)\leq \alpha \}$. Since if $\alpha\leq 0$, then it is easy to show $S$ is bounded, we assume $\alpha>0$ in the following. We need to prove that $S$ is bounded. Now suppose $\bf z=(z_1,...,z_{n+1})\in S$, since $f(\bf z)\geq 0$, we have $\bf g(\bf z)\leq \alpha$. Then there exists some constant $M$ which only depends on $\alpha$ such that $\|\bf w\|\leq M$, where $\bf w=(z_1,...,z_n)$.  Notice that $g(\bf z)\geq 0$ and $\forall 1\leq i\leq p$, $\log(1+\exp(-b_i(\bf a_i^T\bf w+z_{n+1})))\geq 0$, we have \begin{align*} 
\log(1+\exp(\bf a_1^T\bf w+z_{n+1}))\leq \alpha,  \\
\log(1+\exp(-\bf a_2^T\bf w-z_{n+1}))\leq \alpha. 
\end{align*}
This means\begin{align*}
\exp(z_{n+1})\leq(\exp(\alpha)-1)\exp(-\bf a_1^T\bf w)\leq (\exp(\alpha)-1)\exp(M\|\bf a_1\|), \\
\exp(-z_{n+1})\leq(\exp(\alpha)-1)\exp(\bf a_2^T\bf w)\leq (\exp(\alpha)-1)\exp(M\|\bf a_2\|),
\end{align*}
which proves the boundedness of $z_{n+1}$ and completes the proof.
\end{proof} 

(B.4) comes from the fact that $g$ is bounded from below and $\bf B=\bf I$.
For (C.2), first by \cite[Section 2.2]{wang2018convergence}, we know $f(\bf u)+g(\bf v)$ is subanalytic. Moreover $\frac{1}{\gamma}\langle \bf s-\bf u,\bf v-\bf u \rangle+\frac{1}{2\gamma}\|\bf v-\bf u\|^2$ is subanalytic and maps bounded set to bounded set. Hence $\Dg$ is subanalytic as the sum of these two functions by \cite{Xu2013}.

\section{\textbf{Convergence for Physical Simulation Problem}}
\label{appx:PhysicalSimulation}
In this section, we analyze the convergence of Algorithm~\ref{algo1} on the  physical simulation problem~\eqref{eq:OverbyProblem}.
In some cases, Assumption (A.1) would fail to hold. But since (A.1)--(A.3) are only used to prove the decrease of DR envelope, we would show that DR envelope is decreasing even if (A.1) is replaced by weaker assumption. Specifically, it is noted  in~\cite{zhang2019accelerating} that for physical simulation problem \eqref{eq:OverbyProblem}, if $g$ is set to be the hyperelastic energy of StVK material, then $g$ is only locally Lipschitz differentiable and hence doesn't satisfy (A.1).  However, due to the monotone decreasing of DR envelope, we can except that \cite[Theorem 4.1]{themelis2020douglas} still holds in this case, so that the convergence theorem in this paper remains valid. 

In the following, we replace (A.1) by a weaker assumption:\begin{description}
\item[(A.1)'] $\varphi_1$ is Lipschitz differentiable on any bounded set.
\end{description}
Along with this assumption, we further assume:\begin{description}
\item[(A.4)] $\varphi_1$ is level-bounded and $\varphi_2\geq 0$.
\end{description}
To simplify the notation, we define:
\begin{defn}
We define $\lev_{\leq \alpha}\varphi$ to be the set:
\[
\lev_{\leq \alpha}\varphi:=\{\bf x\in\mathbb{R}^n: \varphi(\bf x)\leq \alpha\}.
\] 
\end{defn}
We need the next initial value assumption:\begin{description}
\item[(A.5)] Let $\bf A \bf x_0-\bf B\bf z_0=\bf c$. $\bf y_0$ is chosen such that the augmented Lagrangian function $L(\bf x_0,\bf z_0,\bf y_0)=T_0:=f(\bf  x_0)+g(\bf z_0)<\infty$ and $L(\bf x_1,\bf z_1,\bf y_1)\leq L(\bf x_0,\bf z_0,\bf y_0)$. Assume $\bf s_0=\bf A\bf x_1-\bf y_1$.
\end{description}
Moreover, we need $\gamma$ to be sufficiently small as the next assumption required:
\begin{description}
\item[(A.6)] $\gamma$ is sufficiently small such that $c_0\leq 1$, where 
\[
c_0=\sup\limits_{\lev_{\leq T_0+1}\varphi_1}\frac{\gamma}{2}\|\nabla\varphi_1(\bf x)\|^2.
\]
\end{description}
Here we note that such a $\gamma$ must exist due to (A.4) and the fact that $T_0$ is independent of the choice of $\gamma$.
In the following, we assume $L_1$ to be the Lipschitz modulus of $\nabla\varphi_1$ on the convexhull of the set $\lev_{\leq T_0+1}\varphi_1$.
\begin{lemma}
\label{lemmaH-0}
Suppose that (A.5) holds. Then we have $ \dre\bf (s_0)\leq T_0$.
\end{lemma}
\begin{proof}
Let $\bf u_0=\prox{\gamma\varphi_1}(\bf s_0)$, then by the definition of DR envelope and Proposition~\ref{thm:Equivalence} we can obtain that:
\begin{equation*} 
\dre(\bf s_0)=L(\bf x_1,\bf z_1,\bf y_1)\leq T_0. \qedhere
\end{equation*}
\end{proof}
\begin{lemma}
\label{lemmaH-1}
Assume (A.1)', (A.2)--(A.6) hold and $\gamma<\frac{1}{L_1}$. If it holds that
\[
\varphi_1(\bf u_k)\leq T_0+1,\quad \varphi_1(\bf u_{k+1})\leq T_0+1,
\]
then
\[
\|\bf u_{k+1}-\bf u_{k}\|\leq \frac{1}{1-\gamma L_1}\|\bf s_{k+1}-\bf s_k   \| ,\|\bf u_{k+1}-\bf u_{k}\|\geq \frac{1}{1+\gamma L_1}\|\bf s_{k+1}-\bf s_k   \|.
\]
\end{lemma}
\begin{proof}
By the optimality condition of $\bf u_k$ we know
\[
\gamma\nabla\varphi_1(\bf u_k)+\bf u_k=\bf s_k.
\]
Hence we can infer that
\begin{align*}
\|\bf s_{k+1}-\bf s_k\|&\geq\|\bf u_{k+1}-\bf u_k \|-\gamma\|\nabla\varphi_1(\bf u_{k+1})-\nabla\varphi_1(\bf u_k)   \|,\\
&\geq (1-\gamma L_1)\|\bf u_{k+1}-\bf u_k\|,
\end{align*}
The proof for second part is similar. This completes the proof. 
\end{proof}
\begin{lemma}
\label{lemmaH-2}
Assume (A.1)', (A.2)--(A.6) hold and $\gamma$ is sufficiently small. If for $\bf s_k$ it holds that:
\[ \dre(\bf s_k)\leq T_0,\quad \varphi_1(\bf u_k)\leq T_0+1, \]
then the it also holds for $\bf s_{k+1}$.
\end{lemma}
\begin{proof}
Utilizing the definition of $\bf u_{k+1}$, we obtain that:
\begin{align*}
\varphi_1(\bf u_{k+1})+\frac{1}{2\gamma}\|\bf u_{k+1}-\bf s_{k+1}\|^2\leq \varphi_1(\bf u_k)+\frac{1}{2\gamma}\|\bf u_k-\bf s_{k+1}\|^2.
\end{align*}
By the definition of $\bf s_{k+1}$ and (A.4) we have:
\begin{align*}
&\varphi_1(\bf u_{k+1})+\frac{1}{2\gamma}\|\bf u_{k+1}-\bf s_{k+1}\|^2\\
\leq~& \varphi_1(\bf u_k)+\varphi_2(\bf v_k)+\frac{1}{2\gamma}\|\bf v_k-(2\bf u_k-\bf s_k)\|^2, \\ 
=~& \dre(\bf s_k)+\frac{1}{2\gamma}\|\bf s_k-\bf u_k\|^2, \\
=~&\dre(\bf s_k) + \frac{\gamma}{2}\|\nabla\varphi_1(\bf u_k)\|^2, \\
\leq~& \dre(\bf s_k) + c_0 \leq  \dre(\bf s_k) + 1,
\end{align*}
where we have used the definition of $\dre$ for the first equation, the optimality condition of $\bf u_k$ for the second equation, the definition of $c_0$ for the second inequality. Hence we have proved that 
\[\varphi_1(\bf u_{k+1})\leq \dre(\bf s_k)+1\leq T_0+1. \]
For the estimation of $\dre(\bf s_{k+1})$, we have: 
\begin{align*}
&\dre(\bf s_{k+1})\\
\leq~&\varphi_1(\bf u_{k+1})+\varphi_2(\bf v_k)+\langle \nabla\varphi_1(\bf u_{k+1}),\bf v_k-\bf u_{k+1}\rangle+\frac{1}{2\gamma}\|\bf v_k-\bf u_{k+1}\|^2 ,
\end{align*}
where we have used the definition of DR envelope. We then utilize the definition of $L_1$ and \cite[Lemma 1.2.3]{nesterov2018lectures} to obtain that
\[ \varphi_1(\bf u_{k+1})+\langle\nabla\varphi_1(\bf u_{k+1}),\bf u_k-\bf u_{k+1}\rangle\leq \varphi_1(\bf u_k)+\frac{L_1}{2}\|\bf u_{k+1}-\bf u_k\|^2.            \]
Moreover, we have:
\begin{align*}
&\frac{1}{2\gamma}\|\bf v_k-\bf u_{k+1}\|^2\\
=~&\frac{1}{2\gamma}( \|\bf v_k-\bf u_k\|^2+2\langle \bf v_k-\bf u_k,\bf u_k-\bf u_{k+1}\rangle+\|\bf u_k-\bf u_{k+1}\|^2 ).
\end{align*}
Combing all these three estimation together, we can obtain that:
\[  \dre(\bf s_{k+1})\leq  \dre(\bf s_k)-(\frac{1}{2\gamma}-\frac{L_1}{2}-\gamma L_1^2)\|\bf u_{k+1}-\bf u_k\|^2                          . \]
If $2\gamma^2L_1^2+\gamma L_1<1$, then we have:
\begin{equation*}
\dre(\bf s_{k+1})\leq \dre(\bf s_k) \leq T_0+1.
\qedhere
\end{equation*}
\end{proof}
By induction and Lemma~\ref{lemmaH-2} we can prove the next theorem:
\begin{thm}
\label{thmH-4}
Assume (A.1)', (A.2)--(A.6) hold and $\gamma$ is sufficiently small. Then we have:
\[ \dre(\bf s_{k+1})\leq \dre(\bf s_k)-\left((\frac{1}{2\gamma}-\frac{L_1}{2}-\gamma L_1^2)/(1+L_1\gamma)\right)\|\bf s_{k+1}-\bf s_{k}\|^2.
\]
\end{thm} 
Then all the convergence theorems in this paper can be stated based on Theorem~\ref{thmH-4}. We now verify (A.1)', (A.2)--(A.6) for the physical simulation problem \eqref{eq:OverbyProblem}. In this case, $\varphi_1(\bf x)=f_{\bf A}(\bf x)=f(\bf W^{-1}\bf x)$. So for the case where $f$ is the hyperelastic energy of StVK material, then  $\varphi_1$ satisfies (A.1)' because $f$ satisfies (A.1)'. Notice that $g$ is bounded from below and level-bounded, so $\varphi_2$ is lsc and proper by \cite[Theorem 5.11]{themelis2020douglas}. Moreover, it can be verified that $\varphi_2$ is also level-bounded. So (A.2) is satisfied. (A.3) comes from the lower semi-continuity and level-boundedness of $\varphi$. (A.4) is trivial. (A.5) holds for the choice in \cite[Assumption 3.5]{zhang2019accelerating}.  (A.6) holds for for sufficiently small $\gamma$. Moreover, the aforementioned analysis also shows that (B.2) and (B.4) hold. Finally, since $f,g$ are polynomial and hence semi-algebraic, so $\Dg$ is also semi-algebraic, and then (C.2) follows from Remark~\ref{remark4-5}.
\end{document}